\documentclass[12pt, reqno]{amsart}
\usepackage{amssymb,a4}
\usepackage{amsmath,amsfonts,amssymb,amsthm,amscd}
\usepackage{mathrsfs}

\usepackage{color}
\usepackage{colordvi}

\newcommand{\on}{\operatorname}
\newcommand{\cal}{\mathcal}

\newcommand{\mbf}{\mathbf}

\def\C{{\mathbb C}}

\def\Z{{\mathbb Z}}

\def\N{{\mathbb N}}
\def\1{{\bf 1}}

\def \End{{\rm End}}

\def \<{\langle}
\def \>{\rangle}
\def \w{\omega}

\def \sl{\frak{sl}}

\def \h{\mathfrak{h}}

\def \w{\omega}
\numberwithin{equation}{section}
\newtheorem{theorem}{Theorem}[section]
\newtheorem{prop}[theorem]{Proposition}
\newtheorem{lem}[theorem]{Lemma}
\newtheorem{cor}[theorem]{Corollary}
\theoremstyle{definition}
\newtheorem{defn}[theorem]{Definition}

\newtheorem{remark}[theorem]{Remark}

 \newenvironment{proofof}[1]{\noindent {\bf #1}}

\begin{document}
\title[Heisenberg Vertex Operator Algebras]{Moduli spaces of conformal structures on Heisenberg vertex algebras }

\author{ Yanjun Chu}
%\thanks{}
\address[Chu]{1. School of Mathematics and Statistics,  Henan University, Kaifeng, 475004, China\\
2. Institute of Contemporary Mathematics, Henan University, Kaifeng
475004, China}

\email{chuyj@henu.edu.cn}

\author{Zongzhu Lin}
\address[Lin]{ Department of Mathematics, Kansas State University, Manhattan, KS 66506, USA}
\email{zlin@math.ksu.edu}
\begin{abstract}
This paper is a continuation to understand  Heisenberg  vertex algebras in terms of moduli spaces of their conformal structures.  We study the moduli space of the conformal structures on a Heisenberg vertex algebra that have the standard fixed conformal gradation. As we know in Proposition \ref{pp3.1} in Sect.3, conformal vectors  of the  Heisenberg vertex  algebra  $V_{\hat{\h}}(1,0)$ that have the standard fixed conformal gradation is parameterized by a complex  vector $h$ of its weight-one subspace. First, we classify all  such conformal structures  of the  Heisenberg vertex  algebra  $V_{\hat{\h}}(1,0)$ by describing the automorphism group of the Heisenberg vertex algebra $V_{\widehat{\h}}(1,0)$ and then we describe moduli spaces of their conformal structures that have the standard fixed conformal gradation.  Moreover, we study the moduli spaces of semi-conformal vertex operator subalgebras of each of such conformal structures of  the  Heisenberg vertex  algebra  $V_{\hat{\h}}(1,0)$. 
In such cases,  we describe  their semi-conformal vectors  as pairs consisting
of  regular subspaces  and  the projections of $h$ in these regular subspaces. Then by automorphism groups $G$ of  Heisenberg vertex operator algebras, we get all  $G$-orbits of  varieties consisting of semi-conformal vectors of these vertex operator algebras. Finally, using properties of  these varieties, we give two characterizations  of Heisenberg vertex operator algebras. 
\end{abstract}

\maketitle
\section{Introduction}

%\[\sideset{_n^a}{_m^b}\prod_{\substack{n=1\\m=2}}amndg\]
\subsection{}
 A vertex operator algebra (VOA) is a vertex algebra together with a compatible conformal structure, i.e., a module for the Virasoro Lie algebra. A vertex algebra can have many different conformal structures. Similarly a vertex subalgebra of vertex operator algebra may or may not share the same conformal structure.  This paper is to study the moduli space of the conformal structures on a Heisenberg vertex algebra that have the standard fixed conformal gradation. Then we study the moduli spaces of semi-conformal vertex operator subalgebras of each of the conformal structures. This is a continuation  to characterize Heisenberg vertex operator algebras by the structures of the varieties consisting of  their semi-conformal vectors as in \cite{CL}.
\subsection{}
In  \cite{JL2}, Jiang and the second author  used a class of subalgebras for a vertex operator algebra to study level-rank duality in vertex operator algebra theory. For a vertex operator algebra $V$ with  the conformal vector $\omega$, a vector $\omega'\in V$ is said to be a semi-conformal vector if $\omega'$ is the conformal vector of a vertex operator subalgebra 
$U$ (with possibly different conformal vectors as in \cite[3.11.6]{LL}) such that $\omega_n|_U=\omega'_n|_U$ for all $n\geq 0$ and such a vertex operator subalgebra  $U$ is called a semi-conformal vertex operator subalgebra of $V$(For short, we also call  $U$  a semi-conformal subalgebra  of $V$ throughout the paper.)
To emphasize on the conformal vector $\omega$ of a vertex operator algebra $V$, we write $V$ as  a pair $(V,\omega)$. Let  $\on{Sc}(V,\omega)$ be the set of semi-conformal vectors of $(V,\omega)$. Then the set $\on{Sc}(V,\omega)$  is an affine algebraic variety over $\mathbb{C}$(cf.\cite[Theorem 1.1]{CL}). In this paper,  we shall continue to understand non-standard Heisenberg vertex operator algebras  by using their affine algebraic varieties $\on{Sc}(V,\omega)$.

If $U$ is a semi-conformal subalgebra of  a vertex operator algebra $V$, then its commutant $C_V(U)$ in $V$ is also a semi-conformal subalgebra. We are interested in classifying all semi-conformal subalgebras $U$ such that $ U=C_V(C_V(U))$. Such a semi-conformal  subalgebra  called conformally closed is the unique maximal semi-conformal  subalgebra with a fixed conformal vector.  Each such semi-conformal  subalgebra  is uniquely determined by its conformal vector $\omega'$.  Thus the variety $\on{Sc}(V,\omega)$ also classifies all conformally closed semi-conformal subalgebras in $V$.   
 
 Analog of conformally semi-conformal subalgebra 
 $U$ in a vertex operator algebra $V$ is a von Neumann 
 $C^*$-algebra $M$ in the algebra $B(\cal{H})$ of bounded linear operators on a Hilbert space $\cal{H}$ as in \cite{JS}.  In case 
 $U$ has trivial center (corresponding to factors), by denoting 
 $U'=C_V(U)$,  the pair $(U, U')$ in $ V$ corresponding to the 
 commuting pair of von Neumann algebras $(M, M')$.  
 It is speculated that when $V$ is a rational vertex operator 
 algebra, then any conformally closed semi-conformal 
 subalgebra $U$ and its commutant $U'$ 
 are also rational. One of the question is to decompose $V$ 
 as the sum of modules of $U\otimes U'$.  For the theory of 
 subfactors, the main goal is to construct invariants to classify 
 subfactors of a factor (see \cite{JMS, P, ST}). One of the 
 approaches to the classification in \cite{JMS} is to construct a certain 
 2-module category and some combinatorial data called principal graph. 
 The Heisenberg vertex algebra has a very close analogy to 
 (finite dimensional) Hilbert space of  representation theory of 
 von Neumann algebras. In the up coming work we will 
 construct such analogy for vertex operator algebras using 
 the induction and Jacquet -functors and induced representation 
 theory of semi-conformal subalgebras started in
  \cite{JL3}.

\subsection{}
Let $\h$ be a $d$-dimensional orthogonal space, i.e, $\h$ is a vector space equipped with a nondegenerate symmetric bilinear form $\< \cdot, \cdot\>$.  It is well-known that  a Heisenberg vertex algebra $V_{\widehat{\h}}(1,0)$ has a family conformal vectors $\omega_{h} $ parameterized by a vector $h\in \h$ (see Sect.~\ref{sec3.1}). For different vectors $h$, the resulted vertex operator algebras are not isomorphic in general since the central charges can be different (see \cite[Examples 2.5.9]{BF}), although the underlying vertex algebra is same.  When $h=0$, 
 the  Heisenberg vertex operator algebra is said to be the standard Heisenberg operator algebra and it is easy to compute its automorphism group  which is the complex orthogonal group.  In \cite{CL}, we studied the varieties of  semi-conformal vectors for the standard Heisenberg operator algebra and  characterized its structure. When $h\neq 0$, we say the corresponding  Heisenberg vertex operator algebras  to be non-standard Heisenberg operator algebras. For all non-standard Heisenberg operator algebras, their automorphism groups are no longer the complex orthogonal group. So similar questions are more complicated than standard cases.
In this paper, we shall concentrate on  the moduli spaces of conformal structures of Heisenberg vertex algebras and  semi-conformal structures  of Heisenberg vertex operator algebras.

\subsection{}

Note that $\h(-1)\mbf{1}$ generates a Heisenberg vertex algebra $V_{\widehat{\h}}(1,0)$ of level $1$ with a standard gradation so that  $V_{\widehat{\h}}(1,0)_1=\h(-1)\mbf{1}$. The vertex algebra $V_{\widehat{\h}}(1,0)$ with the conformal vectors $\omega_{h}$  has the conformal gradation being the same as the standard gradation (\ref{e3.1}) in Sect.3. In general for a vertex algebra, different conformal structures have different conformal gradations. One of the questions would be to classify all conformal structures that  have the same conformal gradations. One of the results of this paper is to show that the set $\{\omega_{h}|h\in \h\}$ provides all such possible conformal structures. To study the isomorphism classes of all such conformal structures, we first describe the automorphism group of the Heisenberg vertex algebra $V_{\widehat{\h}}(1,0)$ preserving the standard gradation (\ref{e3.1}) in Sect.3.

 Let  $\on{O}(\h)$ be the group of orthogonal transformations of $\h$. For  the vertex algebra $V_{\widehat{\h}}(1,0)$ and the vertex operator algebras  $(V_{\widehat{\h}}(1,0),\omega_h)$, we determine their automorphism groups $\on{Aut}^{gr}_{VA}(V_{\widehat{\h}}(1,0))$ preserving the standard gradation (\ref{e3.1}) in Sect.3 and $\on{Aut}_{VOA}(V_{\widehat{\h}}(1,0),\omega_h)$, respectively.
\begin{theorem}\label{t1.1}

1)  The automorphism group of the Heisenberg vertex algebra $V_{\widehat{\h}}(1,0)$ preserving the standard gradation (\ref{e3.1}) in Sect.3 is 
$$\on{Aut}^{gr}_{VA}(V_{\widehat{\h}}(1,0))=\on{Aut}(\h,\<\cdot,\cdot\>)\cong \on{O}(\h);$$
2) The automorphism group of the Heisenberg vertex operator algebra $(V_{\widehat{\h}}(1,0),\omega_h)$ for $h\in \h$  is 
\begin{equation}
\on{Aut}_{VOA}(V_{\widehat{\h}}(1,0),\omega_h)=\{\sigma\in \on{O}(\h)|\sigma(h)=h\},
\end{equation}
which is the isotropic subgroup  $\on{O}(\h)_h$ of $\on{O}(\h)$ for  the vector $h\in \h$.
In particular, the automorphism group of $(V_{\widehat{\h}}(1,0),\omega_h)$ for $h=0$
is $$\on{Aut}_{VOA}(V_{\widehat{\h}}(1,0),\omega_h)=\on{Aut}(\h,\<\cdot,\cdot\>)\cong \on{O}(\h);$$\end{theorem}
\begin{cor} Two vertex operator algebras $(V_{\widehat{\h}}(1,0),\omega_h)$ and $(V_{\widehat{\h}}(1,0),\omega_{h'}) $ are isomorphic if and only if $h'=g(h)$ for some $g\in \on{O}(\h)$. In particular the moduli space of the conformal structures on $V_{\widehat{\h}}(1,0)$ preserving the standard gradation (\ref{e3.1}) in Sect.3 is $ \h/\on{O}(\h)$. 
\end{cor}
\subsection{}
Let $\h'\subset \h$ be a subspace of $\h$. If the restriction $\<\cdot, \cdot\>|_{\h'}$ of the bilinear form $\<\cdot, \cdot\>$ on $\h$ to $\h'$ is still nondegenerate, we say $\h'$ is a {\em regular subspace} of $\h$. If $ \h'$ is regular, then we have $ \h=\h' \oplus \h'^{\perp}$ and there is a standard projection map $P_{\h'}: \h\rightarrow \h'$. Set 
\begin{equation}\on{Reg}(\h)=\{\h'\;|\;\h'\; \text{is a regular subspace of } \; \h\}.\end{equation}
Fixing a vector $h\in \h$, we denote by
\begin{equation}\on{Reg}(\h)_{h}=\{(\h',h')|\h'\in \on{Reg}(\h), h'=P_{\h'}(h)\}.\end{equation}
 We can construct an one to one correspondence between $ \on{Sc}(V_{\widehat{\h}}(1,0),\omega_{h})$ and 
 $\on{Reg}(\h)_{h}$. That is, for  each $\omega'\in \on{Sc}(V_{\widehat{\h}}(1,0),\omega_{h})$, there exists   a unique pair $(\h',h')\in\on{Reg}(\h)_{h}$ corresponding to   $\omega'$. 
  
 Note that the automorphism group $G:=\on{Aut}(V_{\widehat{\h}}(1,0),\omega_{h})$  acts on $\on{Reg}(\h)_{h}$ naturally.  Also $\on{Reg}(\h)_{h}$ is a partially ordered set under the subspace inclusion relation. We have the following description of the variety $\on{Sc}(V_{\widehat{\h}}(1,0),\omega_{h})$.
 \begin{theorem} \label{t1.2}
\item [1)] There is an order preserving $G$-equivariant bijection  $\on{Sc}(V_{\widehat{\h}}(1,0),\omega_{h})\cong\on{Reg}(\h)_{h}$;

%\item 2) $\on{Sc}(V_{\widehat{\h}}(1,0),\omega) $ 
 %has exactly $d+1$ orbits under the group $\on{Aut}(V_{\widehat{\h}}(1,0),\omega)$-action and  each $0\leq i \leq d$ %corresponds to the orbit
 %$$\on{Sc}(V_{\widehat{\h}}(1,0),\omega)_i=\{\h'\subset \h| \h'~is~ a ~regular~ subspace~ of ~\h~ with~\on{dim}\h'=i\};$%$
%In particular, $\on{Min}Sc(V_{\widehat{\h}}(1,0),\omega)=Sc(V_{\widehat{\h}}(1,0),\omega)_1.$
\item [2)] There exists a longest chain in $\on{Sc}(V_{\widehat{\h}}(1,0),\omega_{h})$
    such that the length of this chain  equals to $d$:
 there exist $\omega^1,\cdots,\omega^{d-1}\in \on{Sc}(V_{\widehat{\h}}(1,0),\omega_{h})$ such that
 $$0=\omega^0\prec \omega^1\prec\cdots\prec\omega^{d-1}\prec \omega^d=\omega_{h}.$$
\end{theorem}
We denote by $y:=\<h',h'\>$ for $(\h',h')\in \on{Reg}(\h)_h$. By Theorem \ref{t1.2}, we have the $G$- orbit decomposition of $\on{Sc}(V_{\widehat{\h}}(1,0),\omega_{h})$ with respect to the following  two cases.

\begin{theorem}\label{t1.3} For the vector $0\neq h\in \h$ with  $\<h,h\>\neq 0$, all $G$-orbits of $\on{Reg}(\h)_h$ are  as follows
\begin{itemize}
\item[1)] $I_1(k):=\{(\h',h)|\on{dim}\h'=k\}$ for $k=1,\cdots,d$;
\item[2)] $I_2(k):=\{(\h',0)|\on{dim}\h'=k\}$ for $k=0,1,\cdots,d-1$;
\item[3)] When $y\neq 0, \<h,h\>$, $I_3(k,y):=\{(\h',h')| \on{dim}\h'=k,  h'\neq 0, h\}$ for $k=1,\cdots,d-1$;
\item[4)] When $y=\<h,h\>$, $I_4(k,\<h,h\>):=\{(\h',h')| \on{dim}\h'=k, h'\neq 0, h\}$ for $k=1,\cdots,d-2$;
\item[5)] When $y=0$, $I_5(k,0):=\{(\h',h')| \on{dim}\h'=k  , h'\neq 0\}$ for $k=2,\cdots,d-1$,
\end{itemize}
i.e.,  $$\on{Reg}(\h)_{h}=\bigcup_{k=1}^{d}I_1(k)\cup \bigcup_{k=0}^{d-1}I_2(k)\cup\bigcup_{k=1}^{d-1}\left(\bigcup_{y\neq 0,\<h,h\>}I_3(k,y)\right)\cup\bigcup_{k=1}^{d-2}I_4(k,\<h,h\>))\cup\bigcup_{k=2}^{d-1}I_5(k,0).$$
\end{theorem}
\begin{theorem}\label{t1.4} For the vector $0\neq h\in \h$ with  $\<h,h\>= 0$, $\on{Reg}(\h)_{h}$ contains orbits under the action of 
$G$ as follows
\begin{itemize}

\item[1)] $J_1(k):=\{(\h',h)|\on{dim}\h'=k\}$ for $k=2,\cdots,d$;
\item[2)] $J_2(k):=\{(\h',0)|\on{dim}\h'=k\}$ for $k=0,1,\cdots,d-2$;
\item[3)] When $y\neq 0$, $J_3(k,y):=\{(\h',h')| \on{dim}\h'=k\}$ for $k=1,\cdots,d-1$;
\item[4)] When $y=0$, $J_4(k,0):=\{(\h',h')| \on{dim}\h'=k ,h'\neq 0,h\}$ for $k=2,\cdots,d-2$,
\end{itemize}
i.e., $$\on{Reg}(\h)_{h}= \bigcup_{k=2}^{d}J_1(k)\cup\bigcup_{k=0}^{d-2} J_2(k)\cup\bigcup_{k=1}^{d-1}\left(\bigcup_{y\neq 0}J_3(k,y)\right)\cup\bigcup_{k=2}^{d-2}J_4(k,0).$$\end{theorem}
\subsection{}
For each $ \omega' \in \on{Sc}(V_{\widehat{\h}}(1,0),\omega_{h})$,
there is a pair $(\h',h')\in \on{Reg}(\h)_{h}$ corresponding to it. By Theorem 4.6 in Sect.4, we have a  linear transformation $\mathcal{A}_{\omega'}: \h\rightarrow \h $ such that  $\on{Im}\mathcal{A}_{\omega'}=\h', h'=\mathcal{A}_{\omega'}(h)$. And there is another  regular subspace $\on{Ker}\mathcal{A}_{\omega'}$ of $\h$ giving an orthogonal decomposition $\on{Im}\mathcal{A}_{\omega'}\oplus \on{Ker}\mathcal{A}_{\omega'} =\h$ such that $(\on{Ker}\mathcal{A}_{\omega'},h-h')\in  \on{Reg}(\h)_{h}. $ Each of the  abelian Lie algebras $\on{Im}\mathcal{A}_{\omega'}$ and $ \on{Ker}\mathcal{A}_{\omega'}$    generates a Heisenberg vertex operator subalgebra in $V_{\widehat{\h}}(1,0)$. In fact, they are semi-conforaml subalgebras of $(V_{\widehat{\h}}(1,0),\omega_{h})$ and can be both realized as commutant subalgebras of $(V_{\widehat{\h}}(1,0),\omega_{h})$. 
\begin{theorem}\label{t1.5}
For each  $\omega'\in \on{Sc}(V_{\widehat{\h}}(1,0),\omega_{h})$, the following assertions hold. 
\begin{itemize}
\item[1)] $\on{Im}\mathcal{A}_{\omega'}$ generates a Heisenberg vertex operator algebra $$ (V_{\widehat{\on{Im}\mathcal{A}_{\omega'}}}(1,0),\omega')=(C_{V_{\widehat{\h}}(1,0)}(\langle \omega_h-\omega'\rangle),\omega')$$
and $\on{Ker}\mathcal{A}_{\omega'}$ generates a Heisenberg vertex operator algebra $$(V_{\widehat{\on{Ker}\mathcal{A}_{\omega'}}}(1,0),\omega_{h}-\omega')=(C_{V_{\widehat{\h}}(1,0)}(\langle\omega'\rangle),\omega_{h}-\omega');$$

\item[2)]
$$(C_{V_{\widehat{\h}}(1,0)}(V_{\widehat{\on{Ker}\mathcal{A}_{\omega'}}}(1,0)),\omega')=(V_{\widehat{\on{Im}\mathcal{A}_{\omega'}}}(1,0),\omega');$$
$$(C_{V_{\widehat{\h}}(1,0)}(V_{\widehat{\on{Im}\mathcal{A}_{\omega'}}}(1,0))),\omega_{h}-\omega')=(V_{\widehat{\on{Ker}\mathcal{A}_{\omega'}}}(1,0),\omega_{h}-\omega'));$$

\item[3)] $( V_{\widehat{\h}}(1,0),\omega_{h})\cong (C_{V_{\widehat{\h}}(1,0)}(\langle\omega'\rangle),\omega_{h}-\omega')\otimes (C_{V_{\widehat{\h}}(1,0)}(C_{V_{\widehat{\h}}(1,0)}(\langle\omega'\rangle)),\omega').$
\end{itemize}
\end{theorem}
\subsection{}\label{sec1.5}
Based on  above results (See Theorem \ref{t1.2} -Theorem  \ref{t1.5}), we can give two characterizations of  Heisenberg vertex operator algebras $(V_{\widehat{\h}}(1,0),\omega_h)$ for $h\in \h$.

 In this paper, we consider  a vertex operator algebra $(V,\omega)$ satisfying the following conditions:
(1) $V$ is a {\em simple CFT type} vertex operator algebra  generated by $V_1$ (i.e., $V$ is $\N$-graded and $V_0=\C \mbf{1}$);
(2) The symmetric bilinear form $\<u,v\>=u_1v$ for $u,v\in V_1$ is  nondegenerate. For convenience, we call  such a vertex operator algebra $(V,\omega)$   {\em nondegenerate simple CFT type}. We note that for any vertex operator algebra $(V, \omega)$ and any $\omega' \in \on{Sc}(V,\omega)$, one has
$C_{V}(C_{V}\langle\omega'\rangle))\otimes C_{V}(\langle\omega'\rangle)\subseteq V$ as a conformal vertex operator subalgebra.
\begin{theorem} \label{thm1.6}
Let $(V,\omega)$ be a nondegenerate simple CFT type vertex operator algebra. If for
each $\omega'\in \on{Sc}(V,\omega)$, there are
\begin{equation} \label{e5.5}V\cong C_{V}(C_{V}\langle\omega'\rangle))\otimes C_{V}(\langle\omega'\rangle)\end{equation}
then $(V,\omega)$ is isomorphic to the Heisenberg vertex operator algebra $(V_{\hat{\h}}(1,0), \omega_{h})$ with $\h=V_1$ for some $ h\in V_1$. 
\end{theorem}
 For each  $\omega'\in \on{Sc}(V,\omega)$, we can define the height and depth of $\omega'$ in $\on{Sc}(V, \omega)$ analogous to those concepts of prime ideals in a commutative ring. This is also one of the motivations of studying the set of all semi-conformal vectors.  Considering the maximal length of  chains of semi-conformal vectors in $V$,  we can give another characterization of  Heisenberg vertex operator algebras.

\begin{theorem} \label{thm1.7}
Let $(V,\omega)$ be a nondegenerate simple CFT type vertex operator algebra. Assume $\on{\dim}V_1=d$.  If there exists a chain $0=\omega^0\prec\omega^1\prec\cdots \prec \omega^{d-1}\prec \omega^{d}=\omega$
in $\on{Sc}(V,\omega)$ such that $\on{dim}C_{V}(C_{V}(\langle\omega^{i}-\omega^{i-1}\rangle))_1\neq 0, ~for~i=1,\cdots, d$, then  $(V,\omega)$ is isomorphic to the Heisenberg vertex operator algebra $(V_{\widehat{\h}}(1,0), \omega_{h})$ with  $\h=V_1$ for some   $ h\in V_1$. 
\end{theorem}
%Characterizing the Monster Moonshine module by its properties is one of the focuses in conformal field theory. The %FLM conjecture in \cite{FLM} is one of the conjectural characterizations of  the Monster Moonshine module.   

\subsection{} 
One of the main  motivations of this work is to investigate the conformal structure on a vertex subalgebra of a vertex operator algebra.  In conformal field theory, the conformal vector completely determines the conformal structure (the module structure for the Virasoro Lie algebra).  In mathematical physics, a vertex operator algebra has been investigated extensively as a Virasoro module (see \cite{DMZ,D,KL,DLM,LY,M,LS,S,L,Sh}) by virtue of conformal vector.

\subsection{}
 This paper is organized as follows: In Sect.2, we review properties of semi-conformal vectors (subalgebras) of a vertex operator algebra according to \cite{JL2,CL}.  In Sect.3, we show that the set $\{\omega_{h}|h\in \h\}$ provides all conformal structures of the Heisenberg vertex algebra $V_{\widehat{\h}}(1,0)$ that  have the same conformal gradations and  determine the automorphism groups of the Heisenberg vertex algebra $V_{\widehat{\h}}(1,0)$ preserving the standard conformal gradation and Heisenberg vertex operator algebras  $(V_{\widehat{\h}}(1,0),\omega_{h})$.  Then we  study the moduli space of the conformal structures on  that have the standard conformal gradation. In Sect.4,  we describe the moduli spaces of semi-conformal vectors of the  Heisenberg vertex operator algebra $(V_{\widehat{\h}}(1,0),\omega_{h})$ with the vector $h\in \h$ and give the decompositions of orbits of the varieties $\on{Sc}(V_{\widehat{\h}}(1,0),\omega_{h})$. In Sect.5, we study  conformally closed semi-conformal subalgebras of  $(V_{\widehat{\h}}(1,0),\omega_{h})$. In Sect.6, we give two characterizations of  Heisenberg vertex operator algebras $(V_{\widehat{\h}}(1,0),\omega_{h})$ in terms of  properties of their semi-conformal vectors.

{\em Acknowledgement:}  This work started when the first author was visiting Kansas State University from September 2013 to September 2014. He thanks the support by Kansas State University and its hospitality.  The first author is supported by The Key Research Project of Institutions of Higher Education in Henan Province, P.R.China(No. 17A11003) and also thanks China Scholarship Council for their financial supports.  The second author thanks C. Jiang for many insightful discussions. This work was motivated from the joint work with her. The second author also thanks Henan University for the hospitality during his visit in the summer of 2018, during which this work was carried out. 
\section{ Semi-conformal vectors and semi-conformal subalgebras of a vertex operator algebra}
\setcounter{equation}{0}
\subsection{}
For basic notions and results associated with vertex operator algebras, one is
referred  to the books \cite{FLM,LL,FHL,BF}.  We will use $(V, Y, 1)$ to denote a vertex algebra and $(V, Y, 1, \omega)$ for a vertex operator algebra. When we deal with several different vertex algebras, we will use $ Y^V$, $1^{V}$, and $\omega^V$ to indicate the dependence of the vertex algebra or vertex operator algebra $V$.  For example $Y^V(\omega^V, z)=\sum_{n\in \Z} L^V(n)z^{-n-2}$.  To emphasize the presence of the conformal vector $\omega^V$, we will simply write $(V, \omega^V)$  or $(V,\omega) $ for a vertex operator algebra and $V$ simply for a vertex algebra (with $Y^V$ and $1^V$ understood). We refer \cite{BF} for the concept of vertex algebras. Vertex algebras need not be graded, while a vertex operator algebra $(V,\omega^V)$ is always $ \Z$-graded by the $L^V(0)$-eigenspaces $V_n$ with integer eigenvalues $ n \in \Z$. We assume that each $ V_n$ is finite dimensional over $\C$ and $ V_n=0$ if $ n<<0$.   
\subsection{}
 In this section, we shall first review semi-conformal vectors (subalgebras) of a vertex operator algebra (\cite{JL2,CL}).
%Let  $(V, Y^V , 1^V )$ and $(W, Y^W, 1^W)$ be two vertex algebras. It follows  from  Section 3.9 of \cite{LL} that
%a homomorphism $f : V\rightarrow W$ of vertex algebras satisfies
%$$f(Y^V (u, z)v) = Y^W(f(u),z)f(v),~~ \forall u, v \in V;~f(1^V ) = 1^W.$$
%Let $V$ and $W$ be two vertex operator algebras with conformal vectors
%$\omega^V$ and $\omega^W$, respectively. Then  $f$  is called {\em conformal} if $f \circ L^V (n) = L^W(n) \circ f, ~\mbox{for~all}~ n \in\Z,$ i.e., $f(\omega^V ) = \omega^W$.
 %We say $f$  is {\em semi-conformal}
%if $f \circ L^V (n) = L^W(n) \circ f,$ for all $n\geq 0$.
%We remark that, for any vertex algebra homomorphism $f$ between two vertex operator algebras $(V, Y^V, 1^V,\omega^V)$ and $ (W, Y^W, 1^V, \omega^W)$, one always has $f \circ L^V (-1) = L^W(-1) \circ f$. Also $f$ is conformal if and only if $f \circ L^V (n) = L^W(n) \circ f,$ for all $n\geq -2$.  Thus a  semi-conformal  homomorphism $f$ is conformal if and only if $f \circ L^V (-2) = L^W(-2) \circ f $. 
%Let $(V,\omega^V)$ be a vertex subalgebra of $(W,\omega^W)$ and
%the map $f: V \rightarrow W$ is the inclusion, we say $V$ is a
%conformal subalgebra of $W$ if $f$ is
%conformal ($V$ has the same conformal vector with $W$).
%If $f$ is semi-conformal, then $(V,\omega^V)$ is called a semi-conformal
%subalgebra of $(W,\omega^W)$ and $\omega^V$ is called
%a semi-conformal vector of $(W,\omega^W)$.
For a vertex operator algebra $(W,\omega^W)$, we recall that a vertex operator algebra $ (V, \omega^V)$ is called semi-conformal if $ \omega^W_{n}|_V=\omega^V|_V$ for all $ n\geq 0$. We define
$$\begin{array}{lllll}
&\on{ScAlg}(W,\omega^W)=\{(V,\omega^V)\; |\; (V,\omega^V)~ \text{is~a~semi-conformal~subalgebra~of~}(W,\omega^W)\};\\
&\on{Sc}(W,\omega^W)=\{\omega'\in W\; |\;\omega'~\text{is~a~semi-conformal~vector~of}~(W,\omega^W)\};\\
&\overline{\on{S}(W,\omega^W)}=\{(V,\omega')\in \on{ScAlg}(W,\omega^W)|C_{W}(C_{W}(V))=V\},
\end{array}
$$
where $ C_{W}(V)$ is the commutant defined in \cite[3.11]{LL}.  A semi-conformal subalgebra $(U, \omega^U)$ of $W$ is called {\em conformally closed in $W$} if $ C_W(C_W(U))=U$ (see \cite{JL2}). So  the set $\overline{\on{S}(W,\omega^W)}$ consists of all conformally closed
semi-conformal subalgebras of $(W,\omega^W)$.

It follows from the definition that there is a surjective map $ \on{ScAlg}(W, \omega^W)\rightarrow \on{Sc}(W,\omega^W)$ by $(V, \omega^V)\mapsto \omega^V$.  There is also a surjective map $ \on{ScAlg}(W, \omega^W)\rightarrow \overline{\on{S}(W,\omega^W)}$ defined by $(V, \omega^V)\mapsto (C_W(C_W(V), \omega^V)$. Thus, 
 the restriction of the map $\on{ScAlg}(W,\omega^W)\rightarrow \on{Sc}(W, \omega^W)$ to the set $\overline{\on{S}(W,\omega^W)}$  is a bijection (\cite[Proposition 2.1]{CL}).
 
%\begin{proof} The map $\omega'\mapsto C_{W}(C_W(<\omega'>))$ is the inverse map  %$\on{Sc}(W, \omega^W)\rightarrow \overline{\on{S}(W,\omega^W)}$.
%\end{proof}
%Here $C_{W}(V)$ is the commutant of vertex subalgebra $V$ in $W$, defined as
%$$\{u\in W|
%[Y^{W}(u,z_1), Y^{W}(v,z_2)]=0,~\mbox{for ~all~}v\in V\},$$
%or,
%$$\{u\in W|v_nu=0,~\mbox{for~all~}v\in V~\mbox{and}~n\geq 0\}.$$ $C_{W}(V)$ is also
%a vertex subalgebra of $W$ and has commutant $C_{W}(C_{W}(V))$ in $W$(see \cite{LL} for details).

%In a special case that  $(W,\omega^W)$ is a $\N$-graded vertex operator algebra with $W=\coprod\limits_{n\in \N} W_n$ and $W_0=\C\mathbf{1}$, the condition for a  vertex operator  subalgebra to be semi-conformal is much simpler. In this case,  for any vertex subalgebra $V$ of $W$, if $(V,\omega^{V})$ is a vertex operator algebra for  some $\omega^V\in V$, then
%$(V,\omega^{V})$ is a semi-conformal subalgebra of $W$ if and only if
%$\omega^{V}\in W_2\bigcap \mbox{Ker} L^W(1)$(see \cite[Theorem 3.11.12]{LL}).
%Moreover, if $(V,\omega^{V})$ is a semi-conformal subalgebra of $(W,\omega^{W})$,
%then $(C_{W}(V),\omega^{C}=\omega^W-\omega^V)$ is also a semi-conformal
%subalgebra of $(W,\omega^{W})$, that is, if $\omega^{V}\in \on{Sc}(W,\omega^W)$, then $\omega^C=\omega^{W}-\omega^{V}\in \on{Sc}(W,\omega^W)$. Let $(V,\omega^V), (U,\omega^U)$ be two semi-conformal subalgebras of $(W,\omega^W)$.
%If ~$\omega^V=\omega^U$ and $V\subset U$, then we say $(U,\omega^{U})$ is a conformal extension of $(V,\omega^{V})$ in $(W,\omega^{W})$. Moreover,
Let $(V,\omega^V)$ be a semi-conformal subalgebra of $(W,\omega^W)$.
 Then $(V,\omega^{V})$ has a unique maximal conformal extension
$(C_{W}(C_{W}(V)),\omega^{V})$ in $(W,\omega^{W})$ in the sense that
if $(V,\omega^V)\subset (U, \omega^V)$, then $(U, \omega^V)\subset  (C_{W}(C_{W}(V)),\omega^{V})$( see \cite[Corollary 3.11.14]{LL}).  

%\begin{proof} For each $\omega'\in Sc(W,\omega)$, we first look for a semi-conformal subalgebra
%$(U,\omega')\in \overline{S(W,\omega)}$.

%Let $(<\omega'>,\omega')$ be a simple VOA generated by $\omega'$ in $(W,\omega)$. Then $(<\omega'>,\omega')$
%is a semi-conformal subalgebra of $(W,\omega)$ and  $C_{W}(C_{W}(<\omega'>))$
%is the maximal conformal extension of $(<\omega'>,\omega')$ by Corollary 3.11.14 in \cite{15}.
%Denoted by $(U,\omega')=C_{W}(C_{W}(<\omega'>)).$ Because of the uniqueness of  the
%maximal conformal extension of $(<\omega'>,\omega')$, we have $C_{W}(C_{W}(U))=U$, so $(U,\omega')\in \overline{S(W,\omega)}$.

%Conversely, if a semi-conformal subalgebra $(U,\omega')\in \overline{S(W,\omega)}$, we have $\omega'\in Sc(W,\omega)$ by the definition of $\overline{S(W,\omega)}$.
%\end{proof}
\subsection{}
Let  $(W,\omega^W)$ be a general $\Z$-graded vertex operator algebra. The set $\on{Sc}(W,\omega^W)$ forms an affine algebraic variety (\cite[Theorem 1.1]{CL}. In fact,  a semi-conformal  vector $\omega'\in W$ can be  characterized by algebraic  equations of degree at most $2$ as described in \cite[Proposition 2.2]{CL}. The algebraic variety $\on{Sc}(W,\omega^W)$ has also a partial order $\preceq$ (See \cite[Definition 2.7]{CL}), and this partial order  can be characterized by algebraic  equations  as described  in  \cite[Proposition 2.8]{CL}.%\begin{align}\label{2.1}\left\{\begin{array}{llll}
\subsection{}
In fact, the commutant of
$(W,\omega^W)$  can induce an involution $\omega^W-$ of $\on{Sc}(W,\omega^W)$ as follows
 $$
 \begin{array}{lllll}
 \omega^W-:\on{Sc}(W,\omega^W)\longrightarrow \on{Sc}(W,\omega^W)\\
 \hspace{2.8cm}\omega'\longmapsto \omega^W-\omega'.
 \end{array}
 $$
so for  $\omega^1,\omega^2\in \on{Sc}(W,\omega^W)$, we know 
$\omega^W-\omega^1$ and $\omega^W-\omega^2$ are  conformal vectors of commutants $C_{W}(\langle\omega^1\rangle)$ and  $C_{W}(\langle\omega^2\rangle)$, respectively. 
If $\omega^1\preceq \omega^2$, then $\omega^W-\omega^2\preceq \omega^W-\omega^1$.
\section{ Automorphism groups of the Heisenberg vertex algebra $V_{\widehat{\h}}(1,0)$ preserving  the standard gradation and  Heisenberg vertex operator algebras $(V_{\widehat{\h}}(1,0),\omega_h)$ }
\setcounter{equation}{0}
\subsection{} \label{sec3.1}
At first, we  recall  some results of Heisenberg vertex  algebras and refer to \cite{BF, LL} for
more details.

Let $\h$ be a $d$-dimensional vector space with a 
nondegenerate symmetric bilinear form $\<\cdot,\cdot\>$.
$\hat{\h}=\C[t,t^{-1}]\otimes\h\oplus\C C$ is the affiniziation of
the abelian Lie algebra $\h$ defined by
\begin{align*}
[h'\otimes t^{m},\,h''\otimes
t^{n}]=m\<h',h''\>\delta_{m,-n}C\hbox{ and }[C,\hat{\h}]=0
\end{align*}
for  any $h',h'' \in\h,\,m,\,n\in\Z$. Then $\hat{\h}^{\geq
0}=\C[t]\otimes\h\oplus\C C$ is an Abelian subalgebra. For
$\forall \lambda\in\h$, we can define an one-dimensional $\hat{\h}^{\geq
0}$-module $\C e^\lambda$ by the actions $(h\otimes
t^{m})\cdot e^\lambda=\<\lambda,h\>\delta_{m,0}e^\lambda$ and
$C\cdot e^\lambda=e^\lambda$ for $h\in\h$ and $m\geq0$.
Set
\begin{align*}
V_{\widehat{\h}}(1,{\lambda})=U(\hat{\h})\otimes_{U(\hat{\h}^{\geq 0})}\C
e^\lambda\cong S(t^{-1}\C[t^{-1}]\otimes \h),
\end{align*}
which is an $\hat{\h}$-module induced from $\hat{\h}^{\geq 0}$-module 
$\C e^\lambda$. When $ \lambda=0$,  let 
$\1=1\otimes e^0 \in V_{\widehat{\h}}(1,0)$. 
By the strong reconstruction theorem \cite[Theorem. 4.4.1]{BF}, 
there is a unique vertex algebra structure 
$Y:V_{\widehat{\h}}(1,0)\to(\End (V_{\widehat{\h}}(1,{0})))[[z,z^{-1}]]$ 
on $V_{\widehat{\h}}(1,0)$ such that 
\[ Y(h,z)=\sum_{n\in \Z}(h\otimes t^{n})z^{-n-1}\]
with $ h\otimes t^n \in \hat{\h}$ acting on $ V_{\hat{\h}}(1,0)$.

For an orthonormal basis $\{h_1,\cdots, h_d\}$   of $\h$,  
$V_{\widehat{\h}}(1,0)$ has a $\N$-gradation as follow
\begin{equation}\label{e3.1}
V_{\widehat{\h}}(1,0)=\bigoplus_{n\in \N}V_{\widehat{\h}}(1,0)_n,
\end{equation} 
where
$$V_{\widehat{\h}}(1,0)_n=\on{Span}_{\C}\{h_{i_1}(-n_1)\cdots h_{i_k}(-n_k)\cdot\mathbf{1}|k\in \N, 1\leq i_1\leq\cdots
\leq i_k\leq d,\;n_1,\cdots,n_k\geq 1\}.$$

For each $ h \in \h$, we define $\omega_h=\frac{1}{2}\sum\limits_{i=1}^{d}
h_i(-1)^2\cdot\mathbf{1}+h(-2)\mathbf{1}$, then $(V_{\widehat{\h}}(1,0),\,Y,\,\1,\,\w_{h})$ 
is  a  simple vertex operator algebra, with the conformal gradation being the same as defined above. When $h=0$, $(V_{\widehat{\h}}(1,0),\omega_0)$  is called  the {\em standard}
 Heisenberg vertex operator algebra, and it is independent of the choice of the orthonormal basis. When $h\neq 0$ we call $(V_{\widehat{\h}}(1,0),\omega_h)$ {\em non-standard} Heisenberg vertex operator algebra. For each $ \lambda \in \h^*=\h$, 
 $(V_{\widehat{\h}}(1,{\lambda}),Y)$ becomes an irreducible 
 $(V_{\widehat{\h}}(1,0),\omega_h)$-module  
 (see \cite{BF, FLM, LL}).

\begin{prop}  \label{pp3.1}If $ \omega'\in V_{\widehat{\h}}(1,0)_{2}$ is a conformal vector with the same conformal gradation (3.1), then $\omega'=\omega_h$ for some $ h\in \h$. 
\end{prop}

\begin{proof}
Note that 
$$V_{\widehat{\h}}(1,0)_2=\on{Span}_{\C}\{h_{i}(-1)^2\cdot\mathbf{1}, h_{j}(-1)h_{s}(-1)\cdot\mathbf{1}, h_{m}(-2)\cdot\mathbf{1}|1\leq i,m\leq d, 1\leq j<s\leq d\}.$$
Let $\omega'=\sum\limits_{i=1}^{d}a_ih_{i}(-1)^2\cdot\mathbf{1}+\sum\limits_{1\leq i<j\leq d}c_{ij}h_{i}(-1)h_{j}(-1)\cdot\mathbf{1}
+\sum\limits_{i=1}^{d}b_ih_{i}(-2)\cdot\mathbf{1}$.  It follows from\cite{M,L} that $\omega'$ is a conformal vector preserving the 
$\N$-gradation (3.1) of the Heisenberg vertex algebra $ V_{\widehat{\h}}(1,0)$
 if and only if $L'(0)\omega'=2\omega'$ and $L'(0)|_{V_{\widehat{\h}}(1,0)_n}=n\on{Id}_{V_{\widehat{\h}}(1,0)_n}$, where 
 $Y(\omega',z)=\sum\limits_{n\in \Z}L'(0)z^{-n-2}$.
Hence we have
 $$L'(0)=\sum\limits_{i}^{d}a_i\sum\limits_{k\in \Z}:h_i(k)h_i(-k):+\sum\limits_{1\leq i<j\leq d}c_{ij}\sum\limits_{k\in \Z}:h_i(k)h_j(-k):+\sum\limits_{i}^{d}b_ih_i(0),$$
 where  $:\cdots :$ is the normal order.
 
 The equation $L'(0)\omega'=2\omega'$ is equivalent to the equations 
 \begin{eqnarray}
\label{3.3}
\left\{
\begin{array}{llllllll}
4a_{1}^2+c_{12}^2+\cdots+c_{1d}^2=2a_{1};\\
\ \cdots\ \ \ \ \ \cdots\ \ \ \ \ \ \cdots\\
c_{1d}^2+\cdots+c_{d-1d}^2+4a_{d}^2=2a_{d};\\
2a_{1}b_1+c_{12}b_2+\cdots+c_{1d}b_d=b_1;\\
c_{12}b_1+a_{2}b_2+\cdots+c_{2d}b_d=b_2;\\
\cdots\ \ \ \ \ \cdots\ \ \ \ \ \cdots\ \ \ \ \ \cdots\\
c_{1d}b_1+\cdots+c_{d-1d}b_{r-1}+2a_{d}b_d=b_d;\\
c_{i1}c_{1j}+\cdots+c_{ii-1}c_{i-1j}+2a_{i}c_{ij}+c_{ii+1}c_{i+1j}+\cdots+c_{ij-1}c_{j-1j}\\
+2c_{ij}a_{j}+c_{ij+1}c_{j+1j}+\cdots+c_{id}c_{dj}=c_{ij},\mbox{for}~ 1\leq i< j\leq d.\end{array}
\right.
\end{eqnarray}
 Since  $V_{\widehat{\h}}(1,0)$ is generated by 
 $V_{\widehat{\h}}(1,0)_1$ as a vertex algebra, then  
 $L'(0)|_{V_{\widehat{\h}}(1,0)_n}=n\on{Id}_{V_{\widehat{\h}}(1,0)_n}$ 
 is equivalent to 
 $ L'(0)|_{V_{\widehat{\h}}(1,0)_1}=\on{Id}_{V_{\widehat{\h}}(1,0)_1} $. 
 So we have, for each $l$,
 \[ L'(0)h_l(-1)\cdot \mathbf{1}=2a_lh_{l}(-1)\cdot \mathbf{1}
 +\sum_{1\leq i<j\leq d} (c_{il}h_i(-1)\cdot \mathbf{1}
 +c_{lj}h_j(-1)\cdot \mathbf{1})=h_{l}(-1)\cdot \mathbf{1}.\]
 Noting that $ c_{ij}=0$ unless $ i<j$, we get $a_l=\frac{1}{2}$  and 
 $c_{il}=0=c_{lj}$ for $1\leq i<l<j\leq d$. 
 Thus $\omega'$ is a conformal vector preserving the $\N$-gradation (3.1) of  $V_{\widehat{\h}}(1,0)$ if and only if 
 $\omega'=\frac{1}{2}\sum\limits_{i=1}^{d}h_{i}(-1)^2\cdot\mathbf{1}+\sum\limits_{i=1}^{d}b_ih_{i}(-2)\cdot\mathbf{1}$, i.e.,
 there exists a vector $h=\sum\limits_{i=1}^{d}b_ih_i$ such that $\omega'=\omega_0+h(-2)\cdot\mathbf{1}$.
 \end{proof}

\subsection{} 
For each  $h\in \h$, the resulted vertex operator algebra $(V_{\widehat{\h}}(1,0),\omega_h)$ has central charges $  c_{\omega_{h}}=d-12\langle h,h\rangle $ (see \cite[Examples 2.5.9]{BF}) and  thus can be non-isomorphic,  although the underlying vertex algebra 
$(V_{\widehat{\h}}(1,0),\,Y,\,\1)$ is same and they have the same conformal gradation.   By Lemma 3.1, we know the vertex algebra $(V_{\widehat{\h}}(1,0),\,Y,\,\1)$ has infinitely many non-isomorphic conformal structures. More precisely, we classify such conformal vectors up to isomorphism of vertex operator algebras.
\begin{lem}
For $h,h'\in\h$, the vertex operator algebra $(V_{\widehat{\h}}(1,0),\omega_h)$ is isomorphic to the vertex operator algebra $(V_{\widehat{\h}}(1,0),\omega_{h'})$ if and only if  there exists a linear  automorphism $\sigma: \h\rightarrow \h$ preserving the bilinear form $\<\cdot,\cdot\>$  such that $\sigma(h)=h'$.
\end{lem}

\begin{proof}
For $h,h'\in\h$,   if $(V_{\widehat{\h}}(1,0),\omega_h)$ is isomorphic to  $(V_{\widehat{\h}}(1,0),\omega_{h'})$ as vertex operator algebras, i.e., there is an automorphism $\sigma$ of $V_{\widehat{\h}}(1,0)$ such that $\sigma(\mathbf{1})=\mathbf{1}$ and $\sigma (\omega_h)=\omega_{h'}$.  For the orthonormal basis $h_1,\cdots,h_d$ of $\h$, we have 
$$\<\sigma(h_i),\sigma(h_j)\>\mathbf{1}=\sigma(h_i)_{1}\sigma(h_j)\mathbf{1}=\sigma((h_i)_1h_j)\mathbf{1})=\sigma(\<h_i,h_j\>\mathbf{1})=\delta_{ij},$$
where $i,j=1,\cdots, d$ and $\delta_{ij}=1$, if $i=j$; $\delta_{ij}=0$, otherwise. 
Hence the restriction to $V_{\widehat{\h}}(1,0)_1\cong \h$ of  $\sigma$ is an automorphism preserving the bilinear  form $\<\cdot,\cdot\>$ on $\h$  of $\h$. We still denote it by $\sigma$.  It follows from $\sigma (\omega_h)=\omega_{h'}$ that 
we get $\sigma(h)=h'$.

Conversely,  if there exists an automorphism $\sigma$ preserving the bilinear form $\<\cdot,\cdot\>$ on $\h$  of $\h$ such that $\sigma(h)=h'$,  since $V_{\widehat{\h}}(1,0)$ is generated by the subspace $V_{\widehat{\h}}(1,0)_1\cong \h$  as a vertex algebra. By defining $\sigma(\mathbf{1})=\mathbf{1}$,  we can extend $\sigma$ to the whole vertex algebra $V_{\widehat{\h}}(1,0)$.  We still denote it by $\sigma$. By straightly verifying $\sigma(\omega_0)=\omega_0$, we can get $\sigma(\omega_{h})=\omega_{\sigma(h)}=\omega_{h'}$. Thus $\sigma$ is an isomorphism 
from $(V_{\widehat{\h}}(1,0),\omega_h)$ to  $(V_{\widehat{\h}}(1,0),\omega_{h'})$ as vertex operator algebras.
\end{proof}
\begin{proofof}{Proof of Theorem \ref{t1.1}.}  
1) By the proof of the above lemma, we get immediately
the automorphism group of the vertex algebra $V_{\widehat{\h}}(1,0)$ preserving the standard gradation (\ref{e3.1})  is as follow 
$$\on{Aut}^{gr}_{VA}(V_{\widehat{\h}}(1,0))=\on{Aut}\<\h,(\cdot,\cdot)\>\cong \on{O}(\h), $$
where $\on{O}(\h)$ is the group of orthogonal transformations of $\h$.

2) For the standard Heisenberg vertex operator algebra$(V_{\widehat{\h}}(1,0),\,Y,\,\1,\,\w_{0})$,  by 1), we know its automorphism  group  $\on{Aut}(V_{\widehat{\h}}(1,0),\w_{0})=\on{O}(\h)$.

3) Define a map
$$
\begin{array}{llll}
\varphi: \on{Aut}(V_{\widehat{\h}}(1,0),\omega_h)\longrightarrow \on{O}(\h)\\
\hspace{3.4cm} \sigma \mapsto \sigma|_{\h},
\end{array}
$$
then we can check  $\varphi$ is a group isomorphism. Since $(V_{\widehat{\h}}(1,0),\omega_{h})$ is a vertex operator algebra 
generated by its weight-one subspace $\h\mbf{1}=V_{\widehat{\h}}(1,0)_{1}$ (up to isomorphism), then any automorphism $\sigma$ satisfies $\sigma(\h\mbf{1})=\h\mbf{1}$ and is uniquely determined by $\sigma|_{\h}$. Hence different automorphisms have different image under the map $\varphi$. Thus we know $\varphi$ is injective.

On the other hand, any automorphism $\widehat{\sigma}$ can be obtained by extending a $\sigma\in \on{O}(\h)$ and preserving conformal vector $\omega_{h}$. For any $\sigma\in \on{O}(\h)$ and an orthonormal basis $\{h_1,\cdots,h_d\}$ of $\h$, we have 
 $$
\begin{array}{llllll}
\sigma(\omega_{h})=\sigma(\omega_0+h(-2)\mathbf{1})=\omega_0+\sigma(h)(-2)\mathbf{1}
=\w_h.\end{array}
$$Hence   $\sigma(\w_h)=\w_h$ if and only if $\sigma(h)=h$.\qed
\end{proofof}

%\begin{remark}\label{r3.2}
%With respect to another  orthonormal basis  $\{t_1,\cdots,t_d\}$,   $\omega_{\Lambda}$ has a new expression
%$$\omega_{\Lambda}=\frac{1}{2}\sum_{\ell=1}^{d}t_{\ell}(-1)t_{\ell}(-1)\mathbf{1}+\sum\limits_{k=1}^{d}\tilde{\Lambda}_kt_k(-2)\mathbf{1},$$
%where $\tilde{\Lambda}=(\tilde{\Lambda}_1,\cdots, \tilde{\Lambda}_d)$. Let  the  transition matrix from the basis $\{h_1,\cdots,h_d\}$ to the basis $\{t_1,\cdots,t_d\}$ be $o\in \on{O}_{d}(\C)$.
%Then $\Lambda=\tilde{\Lambda}o^{tr}$. For $\sigma\in \on{Aut}(V_{\widehat{\h}}(1,0),\omega_\Lambda)$, its restriction to $\h$, still denoted by $\sigma$, has the corresponding matrix $g=(g_{ij})$ with respect to the basis $\{h_1,\cdots,h_d\}$. Then $g\Lambda^{tr}=\Lambda^{tr}$. With respect to the basis $\{t_1,\cdots,t_d\}$, 
%the restriction $\sigma$ to $\h$  has the matrix $\tilde{g}$. So we have 
%$$\begin{array}{lll}
%\tilde{g}\tilde{\Lambda}^{tr}=o^{tr}go \tilde{\Lambda}^{tr}=o^{tr}goo^{tr}\tilde{\Lambda}^{tr}=o^{tr}g\tilde{\Lambda}^{tr}=o^{tr}\tilde{\Lambda}^{tr}=\tilde{\Lambda}^{tr}.
%\end{array}$$
%Therefore, for $\sigma\in \on{Aut}(V_{\widehat{\h}}(1,0),\omega_\Lambda)$, the condition $\sigma({\omega_{\Lambda}})=\omega_{\Lambda}$ is independent of the choice of 
%orthonormal base of $\h$.
%\end{remark}
\begin{lem} \label{l3.3}  For $h\in \h$ with $\<h,h\>=1$, we have $$\on{Aut}(V_{\widehat{\h}}(1,0),\omega_{h})=\on{O}(\h)_{h}\cong  \{T\in\on{O}_{d}(\C)|
T=
\left(\begin{array}{ll}
1 \ \ 0\ \ \\
0 \ \ T_1
\end{array}
\right), T_1\in \on{O}_{d-1}(\C)\}\cong  \on{O}_{d-1}(\C).
$$
\end{lem}
\begin{proof}
Since $\<h,h\>=1$, it can be extended to an orthonormal basis $h=h_1,h_2,\cdots,h_d$ of $\h$.
Let $T=\left(\begin{array}{llll}
t_{11}\  t_{12}\ \cdots\  t_{1d}\\
t_{21}\  t_{22}\ \cdots\  t_{2d} \\ 
\cdots\ \cdots\ \cdots\ \cdots\\
t_{d1}\ t_{d2}\ \cdots\  t_{dd}
\end{array}
\right)\in \on{Aut}(V_{\widehat{\h}}(1,0),\omega_{h})$. Then $T\in \on{O}_{d}(\C)$ and $$T
\left(\begin{array}{ll}
1\\
0\\
\vdots\\
0
\end{array}
\right)
=\left(\begin{array}{ll}
1\\
0\\
\vdots\\
0
\end{array}
\right).$$
We suppose that $T=\left(\begin{array}{llll}
t_{11}\ \alpha\\
\beta \  \ \ T_1 \end{array}
\right)
$ and $\Lambda=
\left(\begin{array}{llll}
1\\
O
\end{array}
\right)$, where $\alpha=(t_{12},\cdots,t_{1d}), \beta=(t_{21},\cdots,t_{d1})^{tr}$ and $ O=(0,\cdots,0)^{tr}$ is the $d-1$ dimensional column vector.
Using methods of partioned matrices, we have 
$T\Lambda=\left(\begin{array}{llll}
t_{11}\ \alpha\\
\beta \  \ \ t_1 \end{array}
\right)\left(\begin{array}{llll}
1\\
O
\end{array}
\right)=\left(\begin{array}{llll}
t_{11}\\
\beta
\end{array}
\right)=\left(\begin{array}{llll}
1\\
O
\end{array}
\right)$, then $t_{11}=1,\beta=O$.
Since $T\in \on{O}_{d}(\C)$, then $\alpha=O^{tr}$. Hence $T=\left(\begin{array}{llll}
1 \ \ \ O\\
O\  \ \ T_1 \end{array}
\right)
$ and $T_1\in \on{O}_{d-1}(\C)$.
\end{proof}
%\begin{lem}
%Fixing the orthonormal basis $\{h_1,\cdots,h_d\}$ of $\h$, when $\Lambda=(\lambda_1,\cdots,\lambda_d)$, when $\tilde{\Lambda}=a\Lambda, a\neq 0$, $\on{Aut}(V_{\widehat{\h}}(1,0),\omega_{\tilde{\Lambda}})=\on{Aut}(V_{\widehat{\h}}(1,0),\omega_\Lambda)$.
%\end{lem}
%\begin{proof}
%For $\forall g\in \on{Aut}(V_{\widehat{\h}}(1,0),\omega_{\tilde{\Lambda}})$,  when  $g\in  \on{O}_{d}(\C)$,  $g\Lambda^t=\Lambda^t$ if and only if $g\tilde{\Lambda}^t=\tilde{\Lambda}^t$ for $\tilde{\Lambda}=a\Lambda, a\neq 0$, thus 
%$\on{Aut}(V_{\widehat{\h}}(1,0),\omega_{\tilde{\Lambda}})=\on{Aut}(V_{\widehat{\h}}(1,0),\omega_\Lambda)$.
%\end{proof}
\begin{cor}
For any $h\in \h $ with $\<h,h\>\neq 0$,  we have $\on{Aut}(V_{\widehat{\h}}(1,0),\omega_{h})
\cong  \on{O}_{d-1}(\C)$.
\end{cor}
\begin{proof} 
We first choose  an orthonormal basis $\{h_1,\cdots,h_d\}$ of $\h$. For any $h\in \h $ with $\<h,h\>\neq 0$, 
we denote $\xi_1=\frac{h}{\sqrt{\<h,h\>}}$ and $\<\xi_1,\xi_1\>=1$. Then $\xi_1$ can be extended  to an orthonormal basis $\xi_1,\xi_2, \cdots,\xi_d$ of $\h$.  Thus, by  Lemma \ref{l3.3}, we know that $\on{Aut}(V_{\widehat{\h}}(1,0),\omega_{h})=\on{Aut}(V_{\widehat{\h}}(1,0),\omega_{h_1})\cong  \on{O}_{d-1}(\C)$.
\end{proof}

The following lemma is well known in linear algebra, we include a proof for convenience. 
\begin{lem}\label{l4.16}
Assume that $\on{dim}\h>1$,  if two non-zero vectors $\beta,\gamma\in \h$ satisfy  $\< \beta,\beta\>=\<\gamma,\gamma\>=0$, then there exists an orthogonal transformation $\sigma$ of $\h$ such that $\sigma(\beta)=\gamma$.
\end{lem}
\begin{proof} 
  For the nonzero vector $\beta$ with $\<\beta,\beta\>=0$, there exists a vector $\beta'\in \h$  such that $\beta$ and $\beta'$ are linearly independent and $\<\beta',\beta'\>=0, \<\beta,\beta'\>=1$.
Then we get two vectors $\xi_1=\frac{\beta+\beta'}{\sqrt{2}},\xi_2=\frac{\beta-\beta'}{\sqrt{-2}}\in \h$ with  $\<\xi_i,\xi_j\>=\delta_{ij}$ for $i,j=1,2$.  So $\xi_1$ and $\xi_2$ can be extended to an orthonormal basis  $\{\xi_1,\xi_2, \cdots, \xi_d\}$ of $\h$. Similarly, there exists a vector $\gamma'\in \h$ such that $\gamma$ and $\gamma'$ are linearly independent and $\<\gamma',\gamma'\>=0, \<\gamma,\gamma'\>=1$ for the nonzero vector $\gamma$ with $\<\gamma,\gamma\>=0.$ Then we get two vectors $\eta_1=\frac{\gamma+\gamma'}{\sqrt{2}},\eta_2=\frac{\gamma-\gamma'}{\sqrt{-2}}\in \h$ with  $\<\eta_i,\eta_j\>=\delta_{ij}$ for $i,j=1,2$. So $\eta_1$ and $\eta_2$ can be also extended to an orthonormal basis  $\{\eta_1,\eta_2, \cdots, \eta_d\}$ of $\h$. According to the standard theory of linear algebra, we can define an orthogonal transformation $\rho$ of $\h$ such that  $\rho(\xi_s)=\eta_s$ for $s=1,\cdots,d$.  In particular,  $\rho(\beta)=\gamma$.
\end{proof}

Let $C(\h)$ be the cone of all isotropic vectors in $ \h$ (i.e., $ h\in \h$ such that $\< h, h\>=0$). Then $ \on{O}(\h)$ acts on $ C(\h)$ with exactly two orbits $\{0\}$ and $\{c\in C(\h)|c\neq 0\}$. 

$\on{O}(\h)$ acts on the complement $ \h\setminus C(\h)$ and there is an $ \on{O}(\h)$-equivariant map $  \h\setminus C(\h)\rightarrow \on{Reg}^1(\h)$ which is a trivial $ \mathbb{C}^*$-bundle.  Here $\on{Reg}^1(\h)$ is the space of all one-dimensional regular subspaces of $\h$, on which $\on{O}(\h)$ acts transitively. Hence  $\on{Reg}^1(\h)$ is one to one correspondence to $\mathbb{C}^*/\{\pm1\}$. Thus we have the following conclusion 
\begin{theorem} Its isomorphism classes of conformal structures on $V_{\widehat{\h}}(1,0)$ is in one-to-on correspondence to the 
$\{0, c\}\cup \mathbb{C}^*/\{\pm 1\}$. Here $c$ is an nonzero isotropic vector in $\h$. 
\end{theorem}

\section{Semi-conformal vectors of the Heisenberg vertex operator algebra  $(V_{\widehat{\h}}(1,0),\omega_h)$}
\subsection{}
 Note that the computations in \cite[2.5.9]{BF} shows 
that $L^{\omega_{h}}(0)=(\omega_{h})_1$ is always the degree operator. Hence the gradation on $V_{\widehat{\h}}(1,0)$ are the same for all $h\in \h$ and $ V_{\widehat{\h}}(1,0)_{n}$ and  is independent of the choice of $h$. 
Fixing an orthonormal basis $h_1,\cdots,h_d$ of $\h$, we know that $V_{\widehat{\h}}(1,0)_2=\h[-1]\otimes \h[-1]\oplus \h[-2]$  has a basis
\begin{equation}
\label{f3.1}\{h_{i}(-1)h_{j}(-1)\cdot\mathbf{1}; h_{k}(-2)\cdot\mathbf{1}|1 \leq i\leq j\leq d, k=1,\cdots, d\}.\end{equation}
Let \begin{equation}\label{f3.2}\omega'=\sum_{1 \leq i\leq j\leq d}a_{ij}h_{i}(-1)h_{j}(-1)\cdot\mathbf{1}+\sum_{i=1}^{d}b_ih_{i}(-2)\cdot\mathbf{1}\in V_{\widehat{\h}}(1,0)_2.\end{equation} Then there exists a unique symmetric matrix
%$$
\begin{align}\label{f3.3}
A_{\omega'}=
\left(\begin{array}{cccccc}
&2a_{11} &a_{12} &\cdots &a_{1d}\\
&a_{12}& 2a_{22}&\cdots& a_{2d}\\
&\cdots&\cdots&\cdots&\cdots\\
&a_{1d}& \cdots &a_{d-1d} &2a_{dd}
\end{array}\right)
\end{align}
%$$
and a column vector
$B_{\omega'}=(b_1,\cdots,b_d)^{tr}$ with entries in $\C$ such that
\begin{eqnarray}
\begin{array}{llll}
\omega'&=\frac{1}{2}(h_1(-1),\cdots, h_d(-1))A_{\omega'}
\left(\begin{array}{c}
h_1(-1)\\
~~~~~~\vdots\\
 h_d(-1)
\end{array}\right)\cdot\mathbf{1}+(h_1(-2),\cdots,h_d(-2)) B_{\omega'}
\1.
\end{array}
\end{eqnarray}
Assume that $$h=(h_1,\cdots, h_d)
\left(\begin{array}{ll}
\lambda_1\\
\vdots\\
\lambda_d
\end{array}
\right)
=(h_1,\cdots, h_d)\Lambda,
$$we have 
\begin{prop}\label{p3.1}\cite[Proposition 3.1]{CL}
$\omega'\in \on{Sc}(V_{\widehat{\h}}(1,0),\omega_{h})$ if and only if $A_{\omega'},B_{\omega'}$ satisfy
$$A_{\omega'}^{tr}=A_{\omega'},\; A_{\omega'}^2=A_{\omega'},\; A_{\omega'}\Lambda=B_{\omega'},
$$ where $A^{tr}$ is the transpose of  $A$.
\end{prop}
Let $\on{SymId}_{h}=\{(A, B)|A^2=A,A^{tr}=A, A\Lambda=B\}$. Then we have 
\begin{cor}\label{c3.5}
The map $\omega'\mapsto (A,B)$ gives a bijection  between $\on{Sc}(V_{\widehat{\h}}(1,0),\omega_h)$ and $\on{SymId}_{h}$ as sets.
\end{cor}
\begin{prop}
 $\on{Sc}(V_{\widehat{\h}}(1,0),\omega_{h})$ is $\on{Aut}(V_{\widehat{\h}}(1,0),\omega_h)$-invariant.
\end{prop}
\begin{proof}
From Corollary \ref{c3.5}, we can suppose a vector $\omega_{A,B}\in \on{Sc}(V_{\widehat{\h}}(1,0),\omega_{h})$, where 
$(A,B)\in \on{SymId}_{h}$. For arbitrary $\sigma\in \on{Aut}(V_{\widehat{\h}}(1,0),\omega_h)$,  we have 
$\sigma(\omega_{A,B})=\omega_{gAg^{tr}, gB}$, where $g$ is the matrix of  the restriction of $\sigma$  on $\h$ with respect to the basis $\{h_1,h_2,\cdots,h_d\}$.
We can check $(gAg^{tr}, gB)\in \on{SymId}_{h}$, i.e.,  $\omega_{gAg^{tr}, gB}\in \on{Sc}(V_{\widehat{\h}}(1,0),\omega_{h})$.
\end{proof}
% For each fixed  $\Lambda$, we know that the semi-conformal vector $\omega'$ is completely determined  by the matrix $A_{\omega'}$, since $B_{\omega'}=\Lambda A_{\omega'}$.   For simplicity, we will assume that $\Lambda=0$.  For general $ \Lambda$, there is a conformal isomorphism $\phi: (V_{\widehat{\h}}(1,0), \omega_0)\rightarrow (V_{\widehat{\h}}(1,0), \omega_\Lambda)$ inducing a linear isomorphism $ \h\rightarrow \h$ and thus a linear isomorphism $ V_2\rightarrow V_2$ and thus inducing an isomorphism of $ \on{Sc}(V_{\widehat{\h}}(1,0), \omega_0)\rightarrow \on{Sc}(V_{\widehat{\h}}(1,0)$.  Note that $ \phi$ is not a orthogonal linear transformation in general. 
\subsection{}
By Proposition \ref{p3.1}, let $ G_{h}=\on{Aut}(V_{\widehat{\h}}(1,0),\omega_h)$ be the automorphism group of the vertex operator algebra $(V_{\widehat{\h}}(1,0),\omega_h)$. Then $G_{h}$ is a subgroup of $\on{O}(\h)$ (See Theorem \ref{t1.1}) and $G_{h}$ acts on $ V_{\widehat{\h}}(1,0)_n$ for all $n$. In particular, $G_{h}$ acts on the algebraic variety $ \on{Sc}(V_{\widehat{\h}}(1,0),\omega_h)$. One of the questions is to determine the $G_h$-orbits in $\on{Sc}(V_{\widehat{\h}}(1,0),\omega_h)$. In this paper, we will concentrate on the general cases for $h$. 

 With respect to a fixed orthonormal basis $\{ h_1, \cdots, h_d\}$ of $\h$, the symmetric matrix $A_{\omega'}$ defines  a self adjoint  (with respect to the symmetric bilinear form on $\h$)  linear transformation $\mathcal{A}_{\omega'}$ of $\h$. 
 By Proposition \ref{p3.1},  we know that a vector  $\omega'\in \on{Sc}(V_{\widehat{\h}}(1,0),\omega_h)$ is determined by the pair $(A,B)$. Let $\beta$ be the vector of $\h$  with the coordinate $B$ with respect to  the fixed orthonormal basis $\{ h_1, \cdots, h_d\}$ of $\h$.  Then such  a semi-conformal vector $\omega_{A,B}$ is one to one correspondence to a self adjoint  idempotent linear transformation $\mathcal{A}$ of $\h$ satisfying $\mathcal{A}(h)=\beta$. Thus, the set $\on{Sc}(V_{\widehat{\h}}(1,0),\omega_h)$ can be described as the set of  self adjoint idempotent linear transformations of $\h$. Let 
 $$\on{SymId}(\h)=\{\mathcal{A}\in \on{End}(\h)\mid\mathcal{A}^2=\mathcal{A},  \<\mathcal{A}(u),v\>=\<u,\mathcal{A}(v)\>,  \forall u,v, \in \h\}.$$
 $$\on{SymId}(\h)_{h}=\{(\mathcal{A},\beta)\in \on{SymId}(\h)\times \h \mid \mathcal{A}(h)=\beta\}.$$
 Then we have
\begin{prop}\label{p4.1}
The map $\omega_{A,B}\mapsto  ( \mathcal{A},\beta) $ is a bijection from $\on{Sc}(V_{\widehat{\h}}(1,0),\omega_{h})$  to $\on{SymId}(\h)_{h}$.
\end{prop}

Recall the notations  
$\on{Reg}(\h)$ and  $\on{Reg}(\h)_{h}$ as defined in (1.2) and (1.3). When $h=0$, $\on{Reg}(\h)_{h}=\on{Reg}(\h).$ By the standard theory of linear algebra, we have
\begin{lem}\label{l4.2} For each $\mathcal{A}\in \on{SymId}(\h)$,  $\on{Im}\mathcal{A}$ is a regular subspace of $\h$, i.e., $\on{Im}\mathcal{A}\in \on{Reg}(\h)$.
\end{lem}
\begin{prop} \label{c4.3}
 The map $\omega_{A,B}\mapsto (\on{Im}\mathcal{A},\mathcal{A}(h)) $ is a bijection from $\on{Sc}(V_{\widehat{\h}}(1,0),\omega_{h})$ to  $\on{Reg}(\h)_{h}$.
\end{prop}
\begin{proof}
For $\omega_{A,B}\in \on{Sc}(V_{\widehat{\h}}(1,0),\omega_h)$, there is a unique $(\mathcal{A},\beta)\in \on{SymId}(\h)_{h}$.
Hence we have $\h=\on{Im}\mathcal{A}\oplus \on{Ker}\mathcal{A}$ and $\beta=\mathcal{A}(h)\in \on{Im}\mathcal{A}$, where $\on{Ker}\mathcal{A}=\on{Im}\mathcal{A}^{\bot}$. Note that $h=\mathcal{A}(h)+(h-\mathcal{A}(h))$ uniquely in $\h=\on{Im}\mathcal{A}\oplus \on{Ker}\mathcal{A}$. By Lemma \ref{l4.2},  the restriction $\<\cdot,\cdot\>$ on $\h$ to $\on{Im}\mathcal{A}$ is still nondegenerate, then by Proposition \ref{p4.1}, we know that $\omega_{A,B}\mapsto
(\mathcal{A},\beta)$ gives $(\on{Im}\mathcal{A}, \mathcal{A}(h))\in \on{Reg}(\h)_{h}$. Conversely,  for any  $\<\h',h'\>\in \on{Reg}(\h)_{h}$, there exists a projection $\mathcal{A}\in \on{End}(\h)$ such that $\on{Im}(\mathcal{A})=\h'$, $\<\mathcal{A}(u),v\>=\<u,\mathcal{A}(v)\>~for~u,v\in \h$ and $\mathcal{A}(h)=h'$.
Thus, $(\mathcal{A},\mathcal{A}(h))\in \on{SymId}(\h)_{h}$.  By Proposition \ref{p4.1}, we know that there exists a unique
$\omega'\in \on{Sc}(V_{\widehat{\h}}(1,0),\omega_{h})$ such that $(\mathcal{A},\mathcal{A}(h))$ is correspondence to  $\omega'$. Thus, $(\h',h')\mapsto \omega'$ gives the
inverse of the map $\omega'\mapsto (\on{Im}\mathcal{A}, \mathcal{A}(h))$.
\end{proof}

Next, we study a partial order on $\on{SymId}_{h}$.
\begin{prop}\label{p4.4}
Let $\omega_{A_1, B_1},\omega_{A_2,B_2}\in \on{Sc}(V_{\widehat{\h}}(1,0), \omega_{h})$. Here $ (A_1,B_1), (A_2,B_2)\in \on{SymId}_{h}$.  Then $\omega_{A_1, B_1}\preceq \omega_{A_2,B_2}$  if and only if $A_1,B_1,A_2, B_2$ satisfy the following relations
\begin{equation}
A_2A_1=A_1; A_2B_1=B_1=A_1B_2.
\end{equation}
\end{prop}
\begin{proof}
According to \cite[Proposition 2.8]{CL}, we  compute  relations \cite[(2.6)]{CL} to get the following relations:
$L^2(0)\omega_{A_1,B_1}=2\omega_{A_1,B_1}$ can give
\begin{equation}
A_2A_1=A_1, A_2B_1=B_1;
\end{equation}
$L^2(1)\omega_{A_1,B_1}=0$ can give
\begin{equation}
A_2B_1=A_1B_2;
\end{equation}
$L^2(2)\omega_{A_1,B_1}=L^1(2)\omega_{A_1,B_1}$ can give
\begin{equation}
B_2^{tr}B_1=B_1^{tr}B_1;
\end{equation}
$L^2(-1)\omega_{A_1,B_1}=L^1(-1)\omega_{A_1,B_1}$ can give
\begin{equation}
A_2A_1=A_1^2, A_2B_1=A_1B_1;
\end{equation}
The conditions $L^2(n)\omega_{A_1,B_1}=0$ for $n\geq 3$  are satisfied naturally for a CFT-type vertex operator algebra. 
Thus we can obtain $\omega_{A_1, B_1}\preceq \omega_{A_2,B_2}$  if and only if  the following relations hold:
\begin{equation}\label{r4.6}
 A_2A_1=A_1, A_2B_1=B_1,A_2B_1=A_1B_2,B_2^{tr}B_1=B_1^{tr}B_1,A_2A_1=A_1^2, A_2B_1=A_1B_1.
  \end{equation}
 Since $ (A_1,B_1), (A_2,B_2)\in \on{SymId}_{h}$, then  relations (\ref{r4.6}) can be reduced to the following relations
 \begin{equation}
 A_2A_1=A_1, A_2B_1=B_1=A_1B_2.
 \end{equation}
  \end{proof}
  According to the above Proposition, we can get the following conclusions immediately.
 \begin{cor}\label{c4.5}
 Let $(A_1,B_1), (A_2,B_2)\in \on{SymId}_{h}$. Then we can define $ (A_1,B_1)\leq (A_2,B_2) $
 if the relations $ 
 A_2A_1=A_1, A_2B_1=B_1=A_1B_2$ hold. Thus, $\leq$ give a partial order on  $\on{SymId}_{h}$.
 \end{cor}
  
\begin{cor}\label{c4.6} 
Let $(\mathcal{A}_1,\beta_1), (\mathcal{A}_2,\beta_2)\in \on{SymId}(\h)_{h}$. Then we can define $ (\mathcal{A}_1,\beta_1)\leq (\mathcal{A}_2,\beta_2) $
 if the relations $  \mathcal{A}_2\mathcal{A}_1=\mathcal{A}_1, \mathcal{A}_2(\beta_1)=\beta_1=\mathcal{A}_1(\beta_2)$ hold. Thus $\leq$ gives a partial order on  $\on{SymId}(\h)_{h}$.
\end{cor}
\begin{lem}\label{l4.7}
For $(\mathcal{A}_1,\beta_1),(\mathcal{A}_2,\beta_2)\in \on{SymId}(\h)_{h}$, $(\mathcal{A}_1,\beta_1)\leq(\mathcal{A}_2,\beta_2)$ if and only if $\on{Im}\mathcal{A}_1\subset \on{Im}\mathcal{A}_2$ and $\mathcal{A}_2(\beta_1)=\beta_1=\mathcal{A}_1(\beta_2)$.
\end{lem}

\begin{proof} According to Corollary \ref{c4.6}, for
$(\mathcal{A}_1,\beta_1),(\mathcal{A}_2,\beta_2)\in \on{SymId}(\h)_{h}$,  
$(\mathcal{A}_1,\beta_1)\leq(\mathcal{A}_2,\beta_2)$ if and only if the relations $  \mathcal{A}_2\mathcal{A}_1=\mathcal{A}_1, \mathcal{A}_2(\beta_1)=\beta_1=\mathcal{A}_1(\beta_2)$ hold.  Since $\mathcal{A}_2\mathcal{A}_1=\mathcal{A}_1$, then we have
$\on{Im}\mathcal{A}_1=\on{Im}\mathcal{A}_2\mathcal{A}_1.$
And since $\on{Im}\mathcal{A}_2\mathcal{A}_1\subset \on{Im}\mathcal{A}_2$, then $\on{Im}\mathcal{A}_1\subset \on{Im}\mathcal{A}_2$.

Conversely, for  $(\mathcal{A}_1,\beta_1),(\mathcal{A}_2,\beta_2)\in \on{SymId}(\h)_{h}$,
 we have $\h=\on{Ker}\mathcal{A}_1\oplus \on{Im}\mathcal{A}_1=\on{Ker}\mathcal{A}_2\oplus \on{Im}\mathcal{A}_2.$
For $\forall \gamma\in \h$, there is
$$\gamma=\alpha_1+\alpha_2=\gamma_1+\gamma_2,$$
where $\alpha_1\in \on{Ker}\mathcal{A}_1,\alpha_2\in\on{Im}\mathcal{A}_1;\gamma_1\in \on{Ker}\mathcal{A}_2,\gamma_2\in \on{Im}\mathcal{A}_2.$  Since  $\on{Im}\mathcal{A}_1\subset \on{Im}\mathcal{A}_2$, then we have 
$\mathcal{A}_2(\alpha_2)=\alpha_2$. And since $\mathcal{A}_1(\gamma)=\mathcal{A}_1(\alpha_2)=\alpha_2$, 
so we  have
$$
\mathcal{A}_2\mathcal{A}_1(\gamma)=\mathcal{A}_2(\mathcal{A}_1(\alpha_1+\alpha_2))
=\mathcal{A}_2(\mathcal{A}_1(\alpha_2))=\mathcal{A}_2(\alpha_2)=\alpha_2=\mathcal{A}_1(\alpha_2)=\mathcal{A}_1(\gamma).
$$
Hence we get $\mathcal{A}_2\mathcal{A}_1=\mathcal{A}_1$. Since $\mathcal{A}_2(\beta_1)=\beta_1=\mathcal{A}_1(\beta_2)$. Thus we get $(\mathcal{A}_1,\beta_1)\leq (\mathcal{A}_2,\beta_2)$.\end{proof}
Next we can get a partial order on  $\on{Reg}(\h)_{h}$.
 \begin{prop}\label{p4.8}
For any two elements $(\h',h'),(\h'',h'')\in \on{Reg}(\h)_{h}$, we can define 
$(\h',h')\leq (\h'',h'')$ if $\h'\subset \h''$ and $P_{\h''}(h')=h'=P_{\h'}(h'')$, where $P_{\h''}$ is the projection of $\h$ into $\h''$. This gives a partial order on $\on{Reg}(\h)_{h}$.
\end{prop}
\begin{lem}\label{l4.9}
For any two elements $(\h',h'),(\h'',h'')\in \on{Reg}(\h)_{h}$, if $(\h',h')<(\h'',h'')$, then $\on{dim}\h'<\on{dim}\h''$.
\end{lem}
\begin{proof}
According to Proposition \ref{p4.8}, $(\h',h')<(\h'',h'')$ gives $\h'\subset \h''$ and $\h'\neq \h''$, so $\on{dim}\h'<\on{dim}\h''$. 
\end{proof}

\begin{proofof}{\bf Proof of Theorem \ref{t1.2}}  
According to  Corollary \ref{c4.6} and Lemma \ref{l4.7}, we have a bijection preserving orders between $\on{Sc}(V_{\widehat{\h}}(1,0),\omega_{h})$ and $\on{Reg(\h)}_{h}$
by $\omega'\mapsto (\mathcal{A}(\h), \mathcal{A}(h))$.

As we know, $\on{Aut}(V_{\widehat{\h}}(1,0),\omega_{h})$ acts on  $\on{Sc}(V_{\widehat{\h}}(1,0),\omega_{h})$ as follows:
 for  any $\omega_{A,B}\in \on{Sc}(V_{\widehat{\h}}(1,0),\omega_{h})$ and $\sigma\in  \on{Aut}(V_{\widehat{\h}}(1,0),\omega_{h})$,  we have $ \sigma(V_2)=V_2$ and $ \sigma $ preserves the bilinear form $\<\cdot,\cdot\>$ on $\h$. Let $o\in \on{O}_d(\C)$ be the matrix  of $ \sigma$ with respect to a fixed orthonormal basis, then we have $\sigma(\omega_{A,B})=\omega_{oAo^{tr}, oB}\in \on{Sc}(V_{\widehat{\h}}(1,0),\omega_{h})$. Thus $ \on{Im}(oAo^{tr})=\sigma(\on{Im}(A))$.

For a $(\h',h')\in \on{Reg}(\h)_{h}$ and $\sigma\in  \on{Aut}(V_{\widehat{\h}}(1,0),\omega_{h})$, $(\sigma(\h'),\sigma(h'))\in \on{Reg}(\h)_{h}$.
 This gives an action of $\on{Aut}(V_{\widehat{\h}}(1,0),\omega_{h})$ on $\on{Reg}(\h)_{h}$.

According to the actions of  $\on{Aut}(V_{\widehat{\h}}(1,0),\omega_{h})$ on $\on{Sc}(V_{\widehat{\h}}(1,0),\omega_{h})$ and $\on{Reg}(\h)_{h}$, we know that the above bijection also preserves the actions of  $\on{Aut}(V_{\widehat{\h}}(1,0),\omega_{h})$ on $\on{Sc}(V_{\widehat{\h}}(1,0),\omega_{h})$ and $\on{Reg}(\h)_{h}$. So the first assertion 1) holds.

%According to the first assertion  1), we can consider the action of $\on{Aut}(V_{\widehat{\h}}(1,0),\omega)$ on $\on{Sc}(V_{\widehat{\h}}(1,0),\omega)$  by the action of $\on{Aut}(V_{\widehat{\h}}(1,0),\omega)$ on $\on{Reg}(\h)$. Then we have the second assertion.

Let $\h_k=\on{Span}_{\C}\{h_1,\cdots,h_k\}$  and $\alpha_k$ is the projection of $h$ into $\h_k$ for $k=1,\cdots,d$. Then $(\h_1,\alpha_1),\cdots, (\h_1,\alpha_d)\in \on{Reg}(\h)_{h}$ and they satisfy the partial order as follows
$$
(\h_1,\alpha_1)<(\h_2,\alpha_2)<\cdots< (\h_{d-1},\alpha_{d-1})<(\h_{d},\alpha_{d})=(\h,h).
$$According to Lemma \ref{l4.9}, we know this is a longest  chain in  $\on{Reg}(\h)_{h}$ with respect to the partial order in Proposition \ref{p4.8}. Hence the second  assertion holds. \qed
\end{proofof}

\begin{defn}\label{d4.13}
For any two elements $(\h',h'),(\h'',h'')\in \on{Reg}(\h)_{h}$, we say 
 $(\h',h'),(\h'',h'')$ are equivalent, denoted by $(\h',h')\cong(\h'',h'')$ if
 there exists an orthogonal transformation   $\sigma$ of $\h$  fixing the vector $h$ such that $\sigma(\h')=\h''$ and $\sigma(h')=h''$.
\end{defn}
To get the orbits of $\on{Reg}(\h)_{h}$ under the action of $\on{Aut}(V_{\widehat{\h}}(1,0),\omega_{h})$, we need the following lemmas.

According to decomposition theory of orthogonal space, we have 
\begin{lem} \label{l4.21}
If $\h=\h'\oplus\h'^{\bot}$ and $h=h'+h'^{\bot}$ with respect to the nondegenrate bilinear form $\<\cdot,\cdot\>$ on $\h$, then  $(\h',h')\in \on{Reg}(\h)_h$ if and only if  $(\h'^{\bot},h'^{\bot})\in \on{Reg}(\h)_h$. 
\end{lem}
\begin{lem}\label{l4.15}
Let  $(\h',h),(\h'',h)\in \on{Reg}(\h)_{h}$.   Then  $(\h',h)\cong(\h'',h)$ if and only if $\on{dim}\h'=\on{dim}\h''$.
\end{lem}
\begin{proof}
When $\<h,h\>\neq 0$, let $\xi_1=\frac{h}{\sqrt{\<h,h\>}}$.  If $\on{dim}\h'=\on{dim}\h''=k$, then  $\xi_1\in \h'$ can be extended  to an orthonormal basis $\xi_1,\xi_2, \cdots,\xi_k$ of $\h'$ and  $\xi_1\in \h''$ can be extended  to an orthonormal basis $\xi_1,\eta_2,\cdots,\eta_k$ of $\h''$. Define $\rho:\h'\rightarrow \h''$ by $\xi_1\mapsto \xi_1, \xi_i\mapsto \eta_i$ for $i=2,\cdots,k$. It can be extended to  an orthogonal transformation of $\h$ by linearity preserving
non-degenerate bilinear form $\<\cdot,\cdot\>$ on $\h$ and  $\rho(h)=h$. And $\rho$ can generate an automorphism $\widehat{\rho}$ of $(V_{\widehat{\h}}(1,0),\omega_{h}
)$ such that $\widehat{\rho}((\h',h))=(\h'',h)$. Hence $(\h',h)\cong(\h'',h)$.

When $\<h,h\>= 0$ and $h\neq 0$, then we have $\on{dim}\h'=\on{dim}\h''>1$. By Lemma \ref{l4.16}, we get an orthogonal  linear isomorphism $\rho$ from $\h'$ to $\h''$ such that $\rho(h)=h$. It can be extended to  an orthogonal transformation of $\h$ by linearity preserving
non-degenerate bilinear form $\<\cdot,\cdot\>$ on $\h$ and  $\rho(h)=h$. Thus $\rho$ can  generate an automorphism $\widehat{\rho}$ of $(V_{\widehat{\h}}(1,0),\omega_{h}
)$ such that $\widehat{\rho}((\h',h))=(\h'',h)$. Hence $(\h',h)\cong(\h'',h)$.

Conversely,  if $(\h',h)\cong(\h'',h)$,   there exists an automorphism $\widehat{\rho}$ of $(V_{\widehat{\h}}(1,0),\omega_{h})$ such that $\widehat{\rho}((\h',h))=(\h'',h)$, then we have $\on{dim}\h'=\on{dim}\h''$.\end{proof}
\begin{lem}\label{l4.22}
 For  $(\h',h'),(\h'',h)\in \on{Reg}(\h)_{h}$,
when $\on{dim}\h'= \on{dim}\h''$, if $h'\neq h$,   then $(\h',h')\ncong(\h'',h)$. \end{lem}
\begin{proof}
 When $\on{dim}\h'= \on{dim}\h''$,  if $h'\neq h$, then $(\h',h')\ncong(\h'',h)$ since each automorphism $\widehat{\rho}$ of $(V_{\widehat{\h}}(1,0),\omega_{h})$  satisfies $\widehat{\rho}(h)=h$.
\end{proof}
\begin{lem}\label{l4.17}
For  $(\h',h'),(\h'',h'')\in \on{Reg}(\h)_{h}$, if $h', h''\neq 0,h$ and $\on{dim}\h'= \on{dim}\h''$, then  $(\h',h')\cong (\h'',h'')$ if and only if $(h',h')=(h'',h'')$.
\end{lem}
\begin{proof}
If $\<h,h\>\neq 0$, we have the following cases

a) When $\<h',h'\>=\<h'',h''\>\neq 0, \<h,h\>$,
let $\xi_1,\cdots, \xi_k$ be an orthonormal basis of $\h'$ and $\eta_1,\cdots,\eta_k$ be an orthonormal basis 
of $\h''$, respectively. Here $\xi_1=\frac{h'}{\sqrt{\<h',h'\>}}, \eta_1=\frac{h''}{\sqrt{\<h'',h''\>}}$. 
 Define $\rho$ by $\xi_i\mapsto \eta_i$ for $i=1,\cdots,k$. Then it can be extended to a linear map from
 $\h'$ to $\h''$. 
Since  $\h=\h'\oplus \h'^{\bot}=\h''\oplus \h''^{\bot}$, where $\h'^{\bot}$ is the orthonormal complement space of $\h'$ in $\h$, then we can suppose that $h=h'+h'^{\bot}=h''+h''^{\bot}$. Since  $\<h',h'\>=\<h'',h''\>\neq 0, \<h,h\>$, then 
 $\<h'^{\bot},h'^{\bot}\>=\<h''^{\bot},h''^{\bot}\>\neq 0$. Let $\xi_{k+1},\cdots, \xi_d$ be an orthonormal basis of $\h'^{\bot}$ and $\eta_{k+1},\cdots, \eta_d$ be an orthonormal basis of $\h''^{\bot}$. Here $\xi_{k+1}=\frac{h'^{\bot}}{\<h'^{\bot},h'^{\bot}\>}$ and $\eta_{k+1}=\frac{h''^{\bot}}{\<h''^{\bot},h''^{\bot}\>}$. By defining $\xi_{i}\mapsto\eta_i$ for $i=k+1,\cdots,d$, we can extend $\rho$ to an automorphism $\rho$ of $\h$ with preserving the nondegenerate bilinear form $\<\cdot,\cdot\>$. We can find $\rho(h)=h$. For such an automorphism $\rho$
of $\h$, it generates an automorphism $\widehat{\rho}$ of $(V_{\widehat{\h}}(1,0),\omega_{h}
)$,  and $\widehat{\rho}((\h',h'))=(\h'',h'')$. Hence $(\h',h')\cong(\h'',h'')$.

b) When $\<h',h'\>=\<h'',h''\>=\<h,h\>$, 
let $\xi_1,\cdots, \xi_k$ be an orthonormal basis of $\h'$ and $\eta_1,\cdots,\eta_k$ be an orthonormal basis 
of $\h''$, respectively. Here $\xi_1=\frac{h'}{\sqrt{\<h',h'\>}}, \eta_1=\frac{h''}{\sqrt{\<h'',h''\>}}$. 
 Define $\rho'$ by $\xi_i\mapsto \eta_i$ for $i=1,\cdots,k$. Then it can be extended to a linear map from $\h'$ to $\h''$ and $\rho'(h')=h''$. 
Since  $\h=\h'\oplus \h'^{\bot}=\h''\oplus \h''^{\bot}$, where $\h'^{\bot}(\h''^{\bot})$ is the orthonormal complement space of $\h'(\h'')$ in $\h$, then we can suppose that $h=h'+h'^{\bot}=h''+h''^{\bot}$. 
Since  $\<h',h'\>=\<h'',h''\>=\<h,h\>$, then 
 $\<h'^{\bot},h'^{\bot}\>=\<h''^{\bot},h''^{\bot}\>=0$. By Lemma \ref{l4.16},  $\rho'$ can be extended to an orthogonal transformation $\rho$ of $\h$ such that $\rho(h'^{\bot})=h''^{\bot}$.  So $\rho(h)=h$, and it  generates an automorphism $\widehat{\rho}$ of $(V_{\widehat{\h}}(1,0),\omega_{h})$,  and $\widehat{\rho}((\h',h'))=(\h'',h'')$. Hence $(\h',h')\cong(\h'',h'')$.

c) When $\<h',h'\>=\<h'',h''\>=0$,
where $\h'^{\bot}(resp. ~\h''^{\bot})$ is the orthonormal complement space of $\h'(resp. ~\h'')$ in $\h$, then we can suppose that $h=h'+h'^{\bot}=h''+h''^{\bot}$. 
Since  $\<h',h'\>=\<h'',h''\>= 0$, then 
 $\<h'^{\bot},h'^{\bot}\>=\<h''^{\bot},h''^{\bot}\>=\<h,h\>$. 
 As similar as the case b), there exists  an automorphism $\widehat{\rho}$ of $(V_{\widehat{\h}}(1,0),\omega_{h})$,  and $\widehat{\rho}((\h'^{\bot},h'^{\bot}))=(\h''^{\bot},h''^{\bot})$. Hence $(\h',h')\cong(\h'',h'')$.
  
  If $\<h,h\>=0$, we have the following cases
 
 When  $\<h',h'\>=\<h'',h''\>\neq 0$, as similar as the case a),  we can prove  that $(\h',h')\cong(\h'',h'')$. 
 
 When  $\<h',h'\>=\<h'',h''\>=0$,  by Lemma \ref{l4.16}, there exists an orthogonal linear map $\rho_1$ from $\h'$ to $\h''$ such that $\rho_1(h')=h''$. Since $\<h,h\>=0$, then $\<h'^{\bot},h'^{\bot}\>=\<h''^{\bot},h''^{\bot}\>=0$, so by  Lemma \ref{l4.16}, we know that there exists  an orthogonal linear map $\rho_2$ from $\h'^{\bot}$ to $\h''^{\bot}$ such that $\rho_1(h'^{\bot})=h''^{\bot}$. Hence there is an orthogonal transformation  $\rho$ of $\h$ such that  its restriction to $\h'$ is $\rho_1$ and its restriction to $\h'^{\bot}$ is $\rho_2$.
 Hence $\rho(h)=h$ and it can be extended to an automorphism $\widehat{\rho}$ of $(V_{\widehat{\h}}(1,0),\omega_{h})$ such that $\widehat{\rho}((\h',h'))=(\h'',h'')$. Thus, 
 $(\h',h')\cong(\h'',h'')$

Conversely,  when $(\h',h')\cong (\h'',h'')$,   we  get  $\<h',h'\>=\<h'',h''\>$ by Definition \ref{d4.13}. 
\end{proof}
By the standard theory of linear algebra, we have the following lemmas
\begin{lem}\label{l4.18}
For $(\h',0),(\h'',h'')\in \on{Reg}(\h)_{h}$, when $\on{dim}\h'=\on{dim}\h''$,
we have  $(\h',0)\cong (\h'',h'')$ if and only if $h''= 0$.
\end{lem}
\begin{lem}\label{l4.20}
For $(\h',0),(\h'',0)\in \on{Reg}(\h)_{h}$, then  $(\h',0)\cong (\h'',0)$ if and only if  $\on{dim}\h'=\on{dim}\h''$.
\end{lem}

\begin{proofof}{\bf Proof of Theorem \ref{t1.3}}
By Lemma \ref{l4.15}, we know that for $(\h',h)\in \on{Reg}(\h)_{h}$, when $\on{dim}\h'=d$, the orbit determined by $(\h',h)$ contains only one element $(\h,h)$. Considering orbits determined by $(\h',h)\in  \on{Reg}(\h)_{h}$, by Lemma \ref{l4.15} and Lemma \ref{l4.22}, we know that 
$\on{dim}\h'=k$ for $k=1,\cdots, d$ give $d-1$ orbits $I_1(k)$.

Obviously, for $(\h',h)\in \on{Reg}(\h)_{h}$,  when $\on{dim}\h'=0$, the orbit contains only one element 
$(0,0)$. By Lemma \ref{l4.18} and Lemma \ref{l4.20}, considering orbits determined by $(\h',0)\in  \on{Reg}(\h)_{h}$, we know that $\on{dim}\h'=k$ for $k=0,1,\cdots, d-1$ give $d$ orbits $I_2(k)$.

By Lemma \ref{l4.17} ,  considering orbits determined by $(\h',h')\in  \on{Reg}(\h)_{h}$ for $h'\neq 0, h$, when $y:=\<h',h'\>\neq 0, \<h,h\>$ and  $\on{dim}\h'$ takes a value among $1,\cdots, d-1$,  $(\h',h')$ determines only one orbit $I_3(k,y)$.

By Lemma \ref{l4.17} and Lemma \ref{l4.21},  considering orbits determined by $(\h',h')\in  \on{Reg}(\h)_{h}$ for $h'\neq 0, h$ and $y:=\<h',h'\>=\<h,h\>$,  we know that $\on{dim}\h'=k$ for $k=1,\cdots, d-2$ give orbits $I_4(k,y)$.

By Lemma \ref{l4.17} and Lemma \ref{l4.21},  considering orbits determined by $(\h',h')\in  \on{Reg}(\h)_{h}$ for $h'\neq 0$ and $y:=\<h',h'\>=0$,  we know that $\on{dim}\h'=k$ for $k=2,\cdots, d-1$ give orbits $I_5(k,y)$.
\qed
\end{proofof}

\begin{proofof}{\bf Proof of Theorem \ref{t1.4}}
By Lemma \ref{l4.15}, we know that for $(\h',h)\in \on{Reg}(\h)_{h}$, when $\on{dim}(\h')=d$, the orbit determined by $(\h',h)$ contains only one element $(\h,h)$. Considering orbits determined by $(\h',h)\in  \on{Reg}(\h)_{h}$, by Lemma \ref{l4.21}--Lemma \ref{l4.22}, we know that 
$\on{dim}\h'=k$ for $k=2,\cdots, d$ give $d-1$ orbits $J_1(k)$.

Obviously, for $(\h',h)\in \on{Reg}(\h)_{h}$,  when $\on{dim}\h'=0$, the orbit contains only one element 
$(0,0)$. By Lemma \ref{l4.21} and Lemma \ref{l4.18}--Lemma \ref{l4.20}, considering orbits determined by $(\h',0)\in  \on{Reg}(\h)_{h}$, we know that $\on{dim}\h'=k$ for $k=0,1,\cdots, d-2$ give $d-1$ orbits $J_2(k)$.

By Lemma \ref{l4.17},  considering orbits determined by $(\h',h')\in  \on{Reg}(\h)_{h}$, when $y:=(h',h')\neq 0$ and  $\on{dim}\h'$ takes a value among $1,\cdots, d-1$,  $(\h',h')$ determines only one orbit $J_3(k,y)$.

By Lemma \ref{l4.21} and Lemma \ref{l4.17},  considering orbits determined by $(\h',h')\in  \on{Reg}(\h)_{h}$ for $h'\neq 0, h$ and $y:=\<h',h'\>=\<h,h\>=0$,  we know that $\on{dim}\h'=k$ for $k=2,\cdots, d-2$ give orbits $J_4(k,y)$.\qed
\end{proofof}

%\begin{theorem}
%For $\forall \beta\in \h$ with $(\beta,\beta)\neq 0$, we can get an order chain:
%$$(0,0)< (\h^1,\alpha^1)<\cdots<(\h^{d-1},\alpha^{(d-1)})<(\h^{d},\alpha^{(d)})=(\h,\alpha).$$
%\end{theorem}
\section{Conformally closed semi-confromal subalgebras  of the Heisenberg vertex operator algebra $(V_{\widehat{\h}}(1,0),\omega_h)$.}
In this section,  we shall describe the conformally closed semi-conformal subalgebras of the Heisenberg vertex operator algebra $(V_{\widehat{\h}}(1,0), \omega_{h})$.

For $\omega' \in \on{Sc}(V_{\widehat{\h}}(1,0),\omega_{h})$ as defined in \eqref{f3.2} with the corresponding symmetric idempotent matrix $A_{\omega'}$ given in \eqref{f3.3} and a vector $B_{\omega'}=A_{\omega'}\Lambda$.   By Theorem~\ref{t1.2}, we know that
$(\on{Im}\mathcal{A}_{\omega'},\mathcal{A}_{\omega'}(h))\in \on{Reg}(\h)_{h}$ corresponding to $\omega'$.
\begin{prop}
\label{p3.7}
$$\on{Im}\mathcal{A}_{\omega'}=(\on{Ker}_{V_{\widehat{\h}}(1,0)}(L(-1)-L'(-1)))_1\cap V_{\widehat{\h}}(1,0)_1, \; \on{Ker}\mathcal{A}_{\omega'}= (\on{Ker}_{V_{\widehat{\h}}(1,0)} L'(-1))_1\cap  V_{\widehat{\h}}(1,0)_1. $$
\end{prop}
\begin{proof}
By Proposition 3.11.11 and Theorem 3.11.12 in \cite{LL}, we have
$$ \on{Ker}_{V_{\widehat{\h}}(1,0)} L'(-1)=C_{V_{\widehat{\h}}(1,0)}(\<\omega'\>);$$
 $$\on{Ker}_{V_{\widehat{\h}}(1,0)}(L(-1)-L'(-1))=C_{V_{\widehat{\h}}(1,0)}(C_{V_{\widehat{\h}}(1,0)}(\<\omega'\>)).$$
Then $(\on{Ker}_{V_{\widehat{\h}}(1,0)}L'(-1))_{1}\cap V_{\widehat{\h}}(1,0)_1\subset  V_{\widehat{\h}}(1,0)_1\cong \h$.

With  $\omega' \in \on{Sc}(V_{\widehat{\h}}(1,0),\omega_{h})$ being   written as in  \eqref{f3.2},
we have
$$L'(-1)=\sum\limits_{i=1}^{d}a_{ii}\sum\limits_{k\in \Z}:h_i(k)h_i(-1-k):+\sum\limits_{1\leq i<j\leq d}a_{ij}
\sum\limits_{k\in \Z}:h_i(k)h_j(-1-k):,$$
where $:\cdots:$ is the normal order product.
Hence
\begin{equation}\label{3.14} L(-1)-L'(-1)=\sum\limits_{i=1}^{d}(\frac{1}{2}-a_{ii})\sum\limits_{k\in \Z}:h_i(k)h_i(-1-k):-\sum\limits_{1\leq i<j\leq d}a_{ij}
\sum\limits_{k\in \Z}:h_i(k)h_j(-1-k):.\end{equation}

For $\forall h\in V_{\widehat{\h}}(1,0)_{1}$, set $h=\sum\limits_{m=1}^{d}a_mh_{m}(-1)\mathbf{1}$. Using the formula (\ref{3.14}), we have that
$h\in \on{Ker}_{V_{\widehat{\h}}(1,0)}(L(-1)-L'(-1))_1\cap V_{\widehat{\h}}(1,0)_1$ if and only if the vector $(a_1,a_2,\cdots,a_d)$ is a solution of the linear equation system \begin{eqnarray}\label{3.15}\left\{
 \begin{array}{llll}
 (1-2a_{11})x_1-a_{12}x_2-\cdots-a_{1d}x_d=0,\\
 -a_{12}x_1+(1-2a_{22})x_2-\cdots-a_{2d}x_d=0,\\
 \ \cdots\ \ \ \ \cdots\ \ \ \ \cdots\ \ \ \ \cdots\ \ \ \ \cdots\ \ \ \ \cdots\\
 -a_{1d}x_1-a_{2d}x_2-\cdots+(1-2a_{dd})x_d=0.
 \end{array}
 \right.
 \end{eqnarray}
Let $X=(x_1,x_2,\cdots,x_d)^{tr}$. Then the linear equation system
(\ref{3.15}) becomes the matrix equation
$(I-A_{\omega'})X=0,$ where $I$ is the $d\times d$ identity matrix. So we have $\on{Im}\mathcal{A}_{\omega'}=(\on{Ker}_{V_{\widehat{\h}}(1,0)}(L(-1)-L'(-1)))_1\cap V_{\widehat{\h}}(1,0)_1$. Note that  the corresponding matrix for $ \omega_{h}-\omega'$ is $A_{\omega_{h}-\omega'}=I-A_{\omega'}$.   Hence  with $\omega'$ replaced by $\omega_{h}-\omega'$, we have $ \on{Ker}\mathcal{A}_{\omega'}= (\on{Ker}_{V_{\widehat{\h}}(1,0)} L'(-1))_1\cap  V_{\widehat{\h}}(1,0)_1.$
\end{proof}
%\begin{lem}
%For $A\in Var_{A}$, we have
%$$\on{Ker}A=(I-A)(\h).$$
%\end{lem}
%Proof. $\forall \alpha\in Ker A$, then there is
%$A(\alpha)=0$, so we have $I(\alpha)-A(\alpha)=(I-A)(\alpha)=\alpha$, hence $\alpha\in (I-A)(\h)$.

%Conversely, $\forall \beta\in (I-A)(\h)$, then there exists a $\alpha\in \h$ such that
%$\beta=(I-A)(\alpha)$. So $A(\beta) =A(I-A)(\alpha)=(A-A^2)(\alpha)=0$, hence $\beta\in Ker A$.\hspace{5.5cm}$\blacksquare$

%\begin{theorem}
%For each  $\omega'\in \on{Sc}(V_{\widehat{\h}}(1,0),\omega)$, the following are true:
%\begin{itemize}
%\item[1)] In $V_{\widehat{\h}}(1,0)$, $\on{Im}\mathcal{A}_{\omega'}$ 
%generates a Heisenberg vertex operator algebra 
%$$ V_{\widehat{\on{Im}\mathcal{A}_{\omega'}}}(1,0)=
%C_{V_{\widehat{\h}}(1,0)}(<\omega-\omega'>)$$
%and $\on{Ker}\mathcal{A}_{\omega'}$ generates a 
%Heisenberg vertex operator algebra 
%$$V_{\widehat{\on{Ker}\mathcal{A}_{\omega'}}}(1,0)=
%C_{V_{\widehat{\h}}(1,0)}(<\omega'>);$$ 
%
%\item[2)]
%$C_{V_{\widehat{\h}}(1,0)}(V_{\widehat{\on{Ker}\mathcal{A}_{\omega'}}(1,0))
%=V_{\widehat{\on{Im}\mathcal{A}_{\omega'}}}(1,0);\;  
%C_{V_{\widehat{\h}}(1,0)}(V_{\widehat{\on{Im}\mathcal{A}_{\omega'}}}(1,0)))
%=V_{\widehat{\on{Ker}\mathcal{A}_{\omega'}}}(1,0);$
%
%\item[3)] $ V_{\widehat{\h}}(1,0)\cong C_{V_{\widehat{\h}}(1,0)}(<\omega'>)
%\otimes C_{V_{\widehat{\h}}(1,0)}(C_{V_{\widehat{\h}}(1,0)}(<\omega'>)).$
%\end{itemize}
%\end{theorem}

\begin{proofof}{\bf Proof of Theorem 1.6.} By Proposition 5.1, we have
$$C_{V_{\widehat{\h}}(1,0)}(C_{V_{\widehat{\h}}(1,0)}(\<\omega'\>))_1\cong \on{Im}\mathcal{A}_{\omega'}\quad \text{and} \quad C_{V_{\widehat{\h}}(1,0)}(\<\omega'\>)_1\cong\on{Ker}\mathcal{A}_{\omega'}.$$
%Let $V_{\widehat{A(\h)}}(1),V_{\widehat{\on{Ker} A}}(1)$ be the Heisenberg vertex %operator algebra generated
%by $A(\h)$ and $\on{Ker}A$, respectively.
%Assume that
%$$
%A=\left(\begin{array}{llllll}
%&2a_{11} &a_{12} &\cdots &a_{1d}\\
%&a_{12}& 2a_{22}&\cdots& a_{2d}\\
%&\cdots&\cdots&\cdots&\cdots\\
%&a_{1d}& \cdots &a_{d-1d} &2a_{dd}
%\end{array}
%\right)\in \mathfrak{Sc}.
%$$
%Now we write the basis $\{h_1(-1)\mathbf{1},\cdots, h_d(-1)\mathbf{1}\}$ as a column %vector $H=(h_1(-1)\mathbf{1},\cdots, h_d(-1)\mathbf{1})^t$,
%then we have
Since $\on{Im}\mathcal{A}_{\omega'}$  and $\on{Ker}\mathcal{A}_{\omega'}$ are both regular subspaces of $\h$. Then
$$\<\on{Im}\mathcal{A}_{\omega'}\>\cong V_{\widehat{\on{Im}\mathcal{A}_{\omega'}}}(1,0)\quad \text{and} \quad
\<\on{Ker}\mathcal{A}_{\omega'}\>\cong V_{\widehat{\on{Ker}\mathcal{A}_{\omega'}}}(1,0)$$
in $V_{\widehat{\h}}(1,0)$.
Since $A_{\omega'}^2=A_{\omega'}$, we have
\begin{eqnarray}
\begin{array}{llll}
\omega'&=\frac{1}{2}(h_1(-1),\cdots, h_d(-1))A_{\omega'}
\left(\begin{array}{lllll}
h_1(-1)\\
~~~~~~\vdots\\
 h_d(-1)
\end{array}\right)\cdot\mathbf{1}+(h_1(-2),\cdots,h_d(-2))A_{\omega'}\Lambda\mathbf{1}\\
&=\frac{1}{2}
\left(A_{\omega'}
\left(\begin{array}{lllll}
h_1(-1)\\
~~~~~~~~~~~~~~~~~~~~\vdots\\
 h_d(-1)
\end{array}\right)\right)^{tr}
\left(A_{\omega'}
\left(\begin{array}{lllll}
h_1(-1)\\
~~~~~~~~~~~~~~~~~~~~\vdots\\
 h_d(-1)
\end{array}\right)\right)\cdot\mathbf{1}+(h_1(-2),\cdots,h_d(-2))A_{\omega'}\Lambda\mathbf{1}.
\end{array}
\end{eqnarray}
 Then  $\omega'$ is the conformal vector of $\<\on{Im}\mathcal{A}_{\omega'}\>$.

Similarly, we have
$$\begin{array}{llllll}
&\omega_{h}-\omega'=\\
&\frac{1}{2}(h_1(-1),\cdots, h_d(-1))(I-A_{\omega'})
\left(\begin{array}{lllll}
h_1(-1)\\
~~~~~~\vdots\\
 h_d(-1)
\end{array}\right)\mathbf{1}+(h_1(-2),\cdots,h_d(-2))(I-A_{\omega'})\Lambda\mathbf{1}\\
&=\frac{1}{2}(h_1(-1),\cdots, h_d(-1))(I-A_{\omega'})^2
\left(\begin{array}{lllll}
h_1(-1)\\
~~~~~~\vdots\\
 h_d(-1)
\end{array}\right)\mathbf{1}+(h_1(-2),\cdots,h_d(-2))(I-A_{\omega'})\Lambda\mathbf{1}\\
&=\frac{1}{2}
\left((I-A_{\omega'})
\left(\begin{array}{lllll}
h_1(-1)\\
~~~~~~\vdots\\
 h_d(-1)
\end{array}\right)\right)^{tr}
\left((I-A_{\omega'})
\left(\begin{array}{lllll}
h_1(-1)\\
~~~~~~\vdots\\
 h_d(-1)
\end{array}\right)\right)\mathbf{1}\\
&\ +(h_1(-2),\cdots,h_d(-2))(I-A_{\omega'})\Lambda\mathbf{1}.
\end{array}$$
Hence $\omega_{h}-\omega'$  is the conformal vector of $\<\on{Ker}\mathcal{A}_{\omega'}\>$. 

Since $\h=\on{Im}\mathcal{A}_{\omega'}\oplus \on{Ker}\mathcal{A}_{\omega'}$ and $\on{Ker}\mathcal{A}_{\omega'}=(\on{Im}\mathcal{A}_{\omega'})^{\perp}$,
then there is
$$(V_{\widehat{\h}}(1,0),\omega_{h})\cong (\<\on{Im}\mathcal{A}_{\omega'}\>\otimes \<\on{Ker}\mathcal{A}_{\omega'}\>,\omega_{h}).$$
Thus we have \begin{equation}\label{3.17} (C_{V_{\widehat{\h}}(1,0)}(\<\omega'\>),\omega_{h}-\omega')= (V_{\widehat{\on{Ker}\mathcal{A}_{\omega'}}}(1,0),\omega_{h}-\omega')\quad \text{and}\end{equation}
\begin{equation} \label{3.18} (C_{V_{\widehat{\h}}(1,0)}(C_{V_{\widehat{\h}}(1,0)}(\<\omega'\>)),\omega')=(C_{V_{\widehat{\h}}(1,0)}(\<\omega_{h}-\omega'\>),\omega')=(V_{\widehat{\on{Im}\mathcal{A}_{\omega'}}}(1,0),\omega').\end{equation}
Moreover, we have
$$(C_{V_{\widehat{\h}}(1,0)}(V_{\widehat{\on{Ker}\mathcal{A}_{\omega'}}}(1,0)),\omega')=(V_{\widehat{\on{Im}\mathcal{A}_{\omega'}}}(1,0),\omega')$$
and  $$(C_{V_{\widehat{\h}}(1,0)}(V_{\widehat{\on{Im}\mathcal{A}_{\omega'}}}(1,0)),\omega_{h}-\omega')=(V_{\widehat{\on{Ker}\mathcal{A}_{\omega'}}}(1,0),\omega_{h}-\omega').$$
Finally, we can also get that
$$(V_{\widehat{\h}}(1,0),\omega_{h})\cong (C_{V_{\widehat{\h}}(1,0)}(\<\omega'\>)\otimes C_{V_{\widehat{\h}}(1,0)}(C_{V_{\widehat{\h}}(1,0)}(\<\omega'\>)),\omega_{h}).$$\qed
\end{proofof} 
\begin{remark}
From above results, we know that all  of conformally closed semi-conformal subalgebras in $(V_{\widehat{\h}}(1,0),\omega_{h})$  are  Heisenberg vertex operator algebras generated by regular subspaces of the weight-one subspace $V_{\widehat{\h}}(1,0)_1$. 
\end{remark}
\begin{remark}
For each $\omega'\in \on{Sc}(V_{\widehat{\h}}(1,0),\omega_{h})$, we want to
describe the set of all semi-conformal subalgebras with   $\omega'$ being the conformal vector.  Each of such semi-conformal subalgebras is a  conformal extension of $\<\omega'\>$ in $V_{\widehat{\h}}(1,0)$.  We  denote this set  by $\pi^{-1}(\omega')$ which is exactly the set of all conformal subalgebras of the smaller Heisenberg vertex operator algebra $V_{\widehat{\on{Im}\mathcal{A}_{\omega'}}}(1,0)$.  It is an interesting question to determine this set.  This depends on  the $\<\omega'\>$-module structure of  $V_{\widehat{\on{Im}\mathcal{A}_{\omega'}}}(1,0)$.   For affine vertex operator algebras such that the conformal subalgebra $\<\omega'\>$ is rational, then decomposing  each of the members in $\pi^{-1}(\omega')$ as direct sums of irreducible modules was the motivation for the study of semi-conformal subalgebras. 
\end{remark}
\section{Characterizations  of Heisenberg vertex operator algebras}
 In this section, we will use the structures of $\on{Sc}(V,\omega)$ and $\on{ScAlg}(V,\omega)$ to give two characterizations of Heisenberg vertex operator algebras. Let us fix the notation $ Y(u, z)=\sum_{n} u_n z^{-n-1}$ for vertex operators. However, in case  that the vertex operator algebra is defined by a Lie algebra $\h$, we will use the notation $ h(n)=h\otimes t^{n}$ in the algebra $\h[t,t^{-1}]$ with $h \in \h$. 

Let $V$ be a simple $\N$-graded vertex operator algebra  with $V_0=\C 1$. Such  $V$ is also called $a~simple~CFT~type$ vertex operator algebra (\cite{DLMM,DM}).  If $V$ satisfies the further condition
that $L(1)V_1=0$, it is called  {\em strong~CFT~type}. 
Li has shown (\cite{Li}) that such a vertex operator algebra $V$ 
has a unique non-degenerate invariant  bilinear form 
$\langle\cdot,\cdot\rangle$ up to a multiplication of a nonzero scalar. 
In particular, the restriction of $\langle\cdot,\cdot\rangle$ to $V_1$ endows $V_1$ 
with a non-degenerate symmetric invariant  bilinear form  defined by
$u_1(v)=\langle u,v\rangle \mbf{1}$ for $u,v\in V_1$.  For $v\in V_n$, the component operator $v_{n-1}$ is called the zero mode of $v$.  It is well-known that $V_1$ forms a Lie algebra with the bracket operation $[u,v]=u_0(v)$ for $u,v\in V_1$.

  In this paper, we will consider  vertex operator algebras $(V,\omega)$  that satisfy the following conditions:
(1) $V$ is a {\em simple CFT type} vertex operator algebra  generated by $V_1$;
(2) The symmetric bilinear form $\langle u,v\rangle=u_1v$ for $u,v\in V_1$ is  nondegenerate. For convenience, we call  such a vertex operator algebra $(V,\omega)$   {\em nondegenerate simple CFT type}. We note that for any vertex operator algebra $(V, \omega)$, and any $\omega' \in \on{Sc}(V,\omega)$, one has
$C_{V}(C_{V}(\<\omega'\>))\otimes C_{V}(\<\omega'\>)\subseteq V$ as a conformal vertex operator subalgebra.

\begin{lem}\label{l6.1}\label{lem5.1}\cite{CL}
 Let $(V,\omega)$ be a nondegenerate simple CFT type vertex operator algebra. Let 
 $(V', \omega')$ and $(V'', \omega'')$ be two vertex operator subalgebras with possible different conformal vectors. Assume that $(V, \omega)=(V', \omega')\otimes (V'', \omega'')$ is a tensor product of vertex operator algebras (see \cite[3.12]{LL}).  Then
\begin{itemize}
\item [1)] $(V',\omega')$ and $(V'' \omega'')$ are semi-conformal subalgebras and  both are also non-degenerate simple CFT type;
\item [2)] $ V_1=V_1'\otimes \mathbf{1}''\oplus \mathbf{1}'\otimes V_1'',$ is an orthogonal decomposition with
respect to the non-degenerate symmetric bilinear form $\langle\cdot,\cdot\rangle$ on $V_1$;

\item [3)] $[V_1'\otimes \mathbf{1}'',\mathbf{1}'\otimes V_1'']=0$ with the Lie bracket $[\cdot,\cdot]$ on $V_1$;

\item [4)]  $\on{Sc}(V',\omega')\otimes \mbf{1}''$,  $ \mbf{1}'\otimes\on{Sc}(V'',\omega'')$, and $\on{Sc}(V',\omega')\otimes\mbf{1}''+\mbf{1}'\otimes\on{Sc}(V'',\omega'')$ are subsets of $\on{Sc}(V,\omega);$

\item [5)] For each $\widetilde{\omega}'\in \on{Sc}(V',\omega')$, we have
$$C_{V}(\<\widetilde{\omega}'\>\otimes \mbf{1}'')=C_{V'}(\<\widetilde{\omega}'\>)\otimes V''$$
and
$$C_{V}(C_{V}(\<\widetilde{\omega}'\>\otimes \mbf{1}''))=C_{V'}(C_{V'}(\<\widetilde{\omega}'\>))\otimes \mbf{1}''.$$
\end{itemize}
\end{lem}

\begin{lem} \label{lem6.2}\cite{CL}
 Let $(V,\omega)$ be a nondegenerate simple CFT type vertex operator algebra.  If $V=V'\otimes V''$ and $V$ satisfies
\begin{equation}\label{f6.1}
\forall \tilde{\omega}\in \on{Sc}(V,\omega), V\simeq C_{V}(C_{V}(\<\tilde{\omega}\>))\otimes C_{V}(\<\tilde{\omega}\>),
\end{equation}
then $V',V''$ also satisfy (\ref{f6.1}).
Conversely, if $V',V''$ satisfy (\ref{f6.1}) and $\on{Sc}(V,\omega)=\on{Sc}(V',\omega')+\on{Sc}(V'',\omega'')$,
then $V$ also satisfies \eqref{f6.1}.
\end{lem}

\begin{lem}\label{lem5.3}\cite{CL}
If $(V,\omega)$ is a nondegenerate simple CFT type vertex operator algebra satisfying (\ref{f6.1}) and  $\widetilde{\omega} \in \on{Sc}(V,\omega)$ is neither $0$ nor $\omega$, then
$$\on{dim} C_{V}(C_{V}(\<\widetilde{\omega}\>))_1>0\; \;\text{and}\;\; \on{dim} C_{V}(\<\widetilde{\omega}\>)_1>0.$$
\end{lem}
%\begin{proof}
%Since  the vertex operator algebra $(V,\omega)$ satisfies (\ref{f6.1}), we have
%$$\on{dim} C_{V}(C_{V}(<\widetilde{\omega}>))_1+\on{dim} C_{V}(<\widetilde{\omega}>)_1=\on{dim} V_1.$$
%Note that $\widetilde{\omega} \in \on{Sc}(V,\omega)$ is  not zero, then $ C_V(<\tilde{\omega}>)\neq V$. Hence $ C_V(<\tilde{\omega}>)_1\neq V_1$. And since $V$ is generated by $V_1$, then $ C_V(C_V(<\tilde{\omega}>))_1 \neq 0$.  Since $ \tilde{\omega}\neq \omega$, then $ \omega-\tilde{\omega}\neq 0$. Thus   $\on{dim} C_{V}(C_{V}(<\widetilde{\omega}>))_1=\on{dim} C_{V}(<\omega-\widetilde{\omega}>)_1> 0 $. So $\on{dim} C_{V}(C_{V}(<\widetilde{\omega}>))_1>0$ and $\on{dim} C_{V}(<\widetilde{\omega}>)_1>0$.
%\end{proof}
\begin{prop} \label{p5.4}
Let $(V,\omega)$ be a nondegenerate simple CFT type vertex operator algebra.
For $h\in V_1$ with  $ \langle h, h\rangle =1 $ and $ L(1)h=a\mbf{1}$ for  some $a\in \C$, the vertex subalgebra $\<h\>$ of $V$ is isomorphic to $ V_{\hat{\h}}(1,0)$ with $\h=\C h$ and the pair $(\<h\>, \omega')$ with $\omega'=\frac{1}{2}h_{-1}h_{-1}\mbf{1}-\frac{a}{2}h_{-2}\mbf{1}\in V_2$ is a semi-conformal subalgebra of $(V, \omega)$.  In particular,  if $\on{dim}V_1=1$ such that $ L(1)h=a\mbf{1}$ for $h\in V_1$ with $\langle h, h\rangle =1$ and some $a\in \C$, then $(V,\omega)\cong (V_{\widehat{V}_1}(1, 0),\omega')$ with  $\omega'=\frac{1}{2}h_{-1}h_{-1}\mbf{1}-\frac{a}{2}h_{-2}\mbf{1}$.
\end{prop}
\begin{proof}
Let $h \in V_1$ be as in the assumption with $h_1 h=\langle h,h\rangle \mbf{1}=\mbf{1}$. Since $V$ is $ \N-$ graded, we have $h_nh=0$ for all $n\geq 2$ and 
\begin{equation} \label{f5.6}[h_m,h_n]=\sum_{k=0}^{+\infty}
\left(
\begin{array}{lll}
m\\
k
\end{array}
\right)(h_kh)_{m+n-k}
=(h_0h)_{m+n}+m(h_1h)_{m+n-1},
\end{equation}
for $m,n\in \Z$.  Since $[v,u]=v_0(u)$ defines a Lie algebra structure on $V_1$, hence $ h_0h=[h,h]=0$.     So we have
$$[h_m,h_n]=m\langle h,h\rangle \delta_{m+n,0}\on{Id}=m\delta_{m+n,0}\on{Id}.$$
Let $\h=\C h$. This defines  a Heisenberg  Lie algebra homomorphism $\hat{\h}\rightarrow \on{End}(V)$ with $ h\otimes t^n\mapsto h_n$ and $C\mapsto \on{Id}$.(cf. \ref{sec3.1}).  Let $ U=\<h\>$ be the vertex subalgebra generated by $h$. Then  there is a vertex algebra
 homomorphism $V_{\hat{\h}}(1,0) \rightarrow V$ sending $ h(-1)\mbf{1}$ to $ h_{-1}\mbf{1}=h$. The mage of this map is exactly $ U$. Since $V_{\hat{\h}}(1,0)$ is a simple vertex algebra. Thus $U\cong V_{\hat{\h}}(1,0)$.  In the following we discuss the conformal elements. The vertex algebra $V_{\hat{\h}}(1,0)$ with the conformal vector $ \omega'=\frac{1}{2}h(-1)h(-1)\mbf{1}-\frac{a}{2}h_{-2}\mbf{1}$  and the central charge $c=1-3a^2$, where $ L(1)h=a\mbf{1}$ for some $a\in \C$. Its image in $U$ will still be denoted by $ \omega'$ and $ \omega'=\frac{1}{2}h(-1)h(-1)\mbf{1}-\frac{a}{2}h_{-2}\mbf{1}\in V_2$. Thus $(U, \omega')$ is a vertex operator subalgebra of $(V, \omega)$.  
 We now show that this map is semi-conformal or and  $ \omega'$ is semi-conformal in $(V, \omega)$
if $ L(1)h=a\mbf{1}$ (the condition has not yet been used!).  By using \cite[Proposition 2.2]{CL}, we only need to check the following equations, since $(U, \omega')$ is already  a vertex operator algebra. 
\begin{align}\label{5.7}\left\{\begin{array}{llll}
L(0)\omega'=2\omega';\\
L(1)\omega'=0;\\
L(2)\omega'=\frac{1-3a^2}{2}\mbf{1};\\
L'(-1)\omega'=L(-1)\omega';\\
L(n)\omega'=0, n\geq 3.
\end{array}\right.
\end{align}
 The first  is true since $ \omega'\in V_2$. The last one holds since $ V$ is $\N$-graded. The fourth one holds always.  Similar to \eqref{f5.6}, using the assumption that $ L(1)h=a\mbf{1}$ and the fact $L(k)h=0$ for $k\geq 2 $, we have 
 \begin{eqnarray*}\label{f5.8}
 [L(m),h_n]&=&\sum_{k=-1}^{+\infty}
\binom{m+1}{k+1}(L(k)h)_{m+n-k}\\
&=&(L(-1)h)_{m+n+1}+(m+1)(L(0)h)_{m+n}+\binom{m+1}{2}(L(1)h)_{m+n-1}.
\end{eqnarray*}
Using the fact that $ Y(L(-1)v, z)=\frac{d}{dz} Y(v, z)$ one gets $ (L(-1)v)_{n+1}=-(n+1)v_{n}$. Hence 
\[ [L(m),h_n]=-nh_{n+m}+\frac{m(m+1)}{2}a\delta_{m+n,0}.\]
for all $ m, n\in  \Z$. 
Thus 
\begin{eqnarray*}
L(2)\omega'&=&\frac{1}{2}[L(2), h_{-1}^2]\mbf{1}-\frac{a}{2}[L(2),h_{-2}]\mbf{1}\\
&=& \frac{1}{2}([L(2),h_{-1}]h_{-1}+h_{-1}[L(2), h_{-1}])\mbf{1}-\frac{3a^2}{2}\mbf{1}\\
&=& \frac{1}{2}(h_{1}h_{-1}+h_{-1}h_{1})\mbf{1}-\frac{3a^2}{2}\mbf{1}=\frac{1-3a^2}{2}\mbf{1}.
\end{eqnarray*}
\begin{eqnarray*}
L(1)\omega'&=&\frac{1}{2}[L(1), h_{-1}^2]\mbf{1}-\frac{a}{2}[L(1),h_{-2}]\mbf{1}\\
&=& \frac{1}{2}([L(1),h_{-1}]h_{-1}+h_{-1}[L(1), h_{-1}])\mbf{1}-ah_{-1}\mbf{1}\\
&=& \frac{1}{2}(h_{0}h_{-1}+ah_{-1}+h_{-1}h_{0}+ah_{-1})\mbf{1}-ah_{-1}\mbf{1}\\
&=&\frac{1}{2}h_{0}h_{-1}\mbf{1}=\frac{1}{2}h_{0}h\mbf{1}=0.
\end{eqnarray*}
 
  Now assume $ \dim V_1=1$.  Then any orthonormal element $h \in V_1$ satisfies the condition $L(1)h=a\mbf{1}$. Since $V$ is generated by $V_1$, hence $V\simeq <h>=V_{\widehat{V_1}}(1,0)$ and the conformal vector of $V$ is $\frac{1}{2}h_{-1}^2\mbf{1}-\frac{a}{2}h_{-2}\mbf{1}$.  
 \end{proof}
 We remark that a consequence of this proposition is the following which will be repeatedly used in the proof of Theorem~\ref{thm1.6}.
 
 1)  Any $ h \in V_1$ with the property that $\langle h, h\rangle =1$ defines an element $ \omega_h=\frac{1}{2}h_{-1}^2\mbf{1}-\frac{a}{2}h_{-2}\mbf{1}\in \on{Sc}(V, \omega)$ for $L(1)h=a\mbf{1}$;

2) Note that $ C_V(\{h\})=C_V(\<h\>)=C_V(C_V(C_V(\<h\>)))$. If  $ C_{V_1}(h)=\{ v\in V_1 \; |\; h_0(v)=0\} $ is the centralizer of $h$ in the Lie algebra $V_1$ and $ h^\perp =\{ v\in V_1 \; |\; \langle v, h\rangle =0\} $ is the subspace, then $ C_{V_1}(h)\cap h^\perp \subseteq (C_V(\<h\>))_1$.  In particular, let  $V_1$ be an abelian Lie algebra and $ \dim V_1\geq 2$. Then $ \on{Sc}(V, \omega)\setminus \{0, \omega\}$ is not empty.  

%3) If $ (V, \omega)$ has central charge $\neq 1$, then for any $ h\in V_1$ as in %Proposition~\ref{p5.4}, then $ C_V(C_V(\<h\>))\neq V$. Hence, $ \on{Sc}(V, \omega)
%\setminus \{0, \omega\}$ is not empty. 

%4) For any $ \omega'\in \on{Sc}(V, \omega)$, if $ L(1)V_1=0$, then $ L'(1)(C_V(C_V(<%%\omega'>)))_1=0$.
 
%\begin{theorem}
%Let $(V,\omega)$ be a nondegenerate simple CFT type vertex operator algebra.
%\item 1) If $\on{dim}V_1=1$, then $V$ is 
%the Heisenberg vertex operator algebra generated by $V_1$;
%\item 2) If $\on{dim} V_1\geqslant 2$, $\on{Sc}(V,\omega)\neq \{0, \omega\}$ and for
%each $\omega'\in Sc(V,\omega)$, there are
%\begin{equation} \label{e5.5}
%V\cong C_{V}(C_{V}(<\omega'>))\otimes C_{V}(<\omega'>).\end{equation}
%Then $V$ is isomorphic to the Heisenberg vertex operator algebra generated by $V_1$.
%\end{theorem}
\begin{proofof}{\bf Proof of Theorem 1.7.} We will use induction on $\dim V_1$.  If $\on{dim}V_1=1$, there exists a $h\in V_1$ with $\langle h,h\rangle=1$ and   $ L(1)h=a\mbf{1}$ for some $a\in \C$, the theorem is  the special case in Proposition~\ref{p5.4}. Hence $(V,\omega)\cong V_{\widehat{V}_1}(1,0)$ with  $\omega'=\frac{1}{2}h_{-1}h_{-1}\mbf{1}-\frac{a}{2}h_{-2}\mbf{1}$.

Assume  $\on{dim}V_1\geq 2$. Since the bilinear form $ \langle\cdot, \cdot \rangle $ on $ V_1$ is non-degenerate, there is an $ h\in V_1$ such that $\langle h, h \rangle=1$. If $ L(1)h=a\mbf{1}$ for some $a\in \C$, Proposition~\ref{p5.4} implies that $\omega'=\omega_h=\frac{1}{2} h_{-1}^2\mbf{1}-\frac{a}{2}h_{-2}\mbf{1}\in V_2$ is a semi-conformal vector of $(V,\omega)$. 
Then the assumption of the theorem implies 
\begin{equation} 
\label{f5.9} (V, \omega)\cong (C_V(\<\omega_h\>), \omega-\omega_h)\otimes (C_V(\<\omega-\omega_h\>), \omega_h)
\end{equation}
as vertex operator algebras. 

Let $ V'(h)=C_V(\<\omega-\omega_h\>)$ with conformal vector $\omega_h$ and $ V''(h)=C_V(\<\omega_h\>)$ with conformal vector $ \omega-\omega_h$.  Then, by Lemma~\ref{l6.1} and Lemma~\ref{lem6.2}, both $(V'(h), \omega_h)$ and $(V''(h), \omega-\omega_h)$ satisfy the conditions of this theorem. If we can prove that  
\begin{equation}\label{f5.10}
\dim V'(h)_1<\dim V_1\quad \text{ and } \quad \dim V''(h)_1<\dim V_1,
\end{equation}
 then the induction assumption will imply that both $V'(h)$ and $V''(h)$ are Heisenberg vertex operator algebras of the specified rank and of the specified conformal vector. Thus the tensor product vertex operator algebra also has the same property.   By Lemma~\ref{lem5.3}, we only need to show that $ \omega_h\neq \omega $ when $ \dim V_1 \geq 2$. However one cannot be sure that  $ \omega_h\neq \omega$ in this case for arbitrarily chosen $ h\in V_1$.  But we can choose  $h\in V_1$ with $ \langle h, h\rangle =1$ such that $ \dim V'(h)_1$ is the smallest possible. We claim that $ \dim V'(h)_1=1$ and then $ \dim V''(h)_1=\dim V_1-1>0$. Then we are done. In the following we prove this claim. 
 
 Note that $ h\in V'(h)_1$ implies that $ \dim V'(h)_1>0$. Assume that $ \dim V'(h)_1\geq 2$. By Lemma~\ref{lem5.1}.1),  $ V'(h)$ is nondegenerate simple CFT type and is generated by $V'(h)_1$ as a vertex algebra. Thus the bilinear form $\langle\cdot, 
 \cdot \rangle$ remains nondegenerate when restricted to $ V'(h)_1$. There is an $ h'\in 
 V'(h)_1$ such that $ \langle h', h'\rangle =1$ and $ \langle h, h'\rangle =0$. Hence $h$ and $ h'$ are linearly independent. Set $ \omega_{h'}=\frac{1}{2}h'_{-1}h'_{-1}\mbf{1}-\frac{b}{2}h'_{-2}\mbf{1}\in V'(h)_2$ for $L(1)h'=b\mbf{1}$. Then $\omega_{h'}$ is a semi-conformal vector of $V$ by Proposition~\ref{p5.4}. By using \cite[Proposition 2.2]{CL}, $\omega_{h'}$ is also a semi-conformal vector of $V'(h)$. Thus $ \omega_{h'}\preceq \omega_{h}$. 
 We now claim that $\omega_{h'}\neq \omega_{h}$. Applying Lemma~\ref{lem5.3}, we will have $\dim C_{V'(h)}(C_{V'(h)}(\<\omega_{h'}\>))_1<\dim V'(h)_1$. 
Since  $ C_{V'(h)}(C_{V'(h)}(\<\omega_{h'}\>))=C_V(C_V(\<\omega_{h'}\>))=V'(h')$, then 
we have $\dim V'(h')_1<\dim V'(h)_1$. So we have a contradiction to the choice of $h$. 
Hence we must have $ \omega_h=\omega_{h'}$. Note that $V'(h)$ contains 
to vertex subalgebras $U(h)$ and $U(h')$ generated by $ h$ and $h'$ 
respectively with $ \omega_h$ and $\omega_{h'}$ being conformal 
vectors respectively as in the proof of Proposition~\ref{p5.4}  and $ h'_{1}\omega_{h'}=h'$ and $ h_{1}\omega_{h}=h$. But using 
\[ [h'_1, h_{-1}]=(h'_0h)_0+(h'_1h)_{-1}=(h'_0h)_0 \quad \text{and } \quad v_{n}\mbf{1}=0 \; \text{ for all } n\geq 0, \]
we compute 
\begin{eqnarray*}
h'&=&h'_{1}\omega_{h'}=h'_1\omega_{h}=\frac{1}{2}( [h'_1, h_{-1}]h_{-1}+h_{-1}[h'_1,h_{-1}])\mbf{1}\\
&=& \frac{1}{2}(h'_0h)_0 h_{-1}\mbf{1}=\frac{1}{2}(h'_{0}h)_0h=\frac{1}{2}[[h', h],h].
\end{eqnarray*}
 By setting $ h''=[h', h]$,  we have $ [h'',h]=2h'$ in the Lie algebra $ V_1$. 
 Exchanging $ h$ and $h'$ we also get $ [[h, h'], h']=2h$, i.e., $ [h'', h']=-2h$.   This also implies that $ h''\neq 0$ and  the linear span of   $\{h, h', h''\}$ is Lie algebra isomorphic to $ \frak{sl}_2$ with  the standard generators 
 \begin{equation} e= \frac{h+ih'}{\sqrt{2}i}, \quad f= \frac{ih+h'}{\sqrt{2}i}, \quad k=ih''.
  \end{equation}
  Thus we have 
  \[[ k,e]=2e, \quad [k,f]=-2f, \quad [e,f]=k.\]
  Since the bilinear form $ \langle \cdot, \cdot \rangle $ is invariant, a direct computation shows that 
  \[\langle h'', h\rangle=\langle h'', h'\rangle =0 \quad \text{and} \quad \langle h'', h''\rangle=-2.\] 
  Then $ \omega_{\frac{h''}{\sqrt{-2}}}=-\frac{1}{4}h''_{-1}h''_{-1}\mbf{1}-\frac{c}{2\sqrt{-2}}h''_{-2}\mbf{1}$  for $L(1)h''=c\mbf{1}$ is also a semi-conformal vector of $ V'(h)$. If $ \omega_{\frac{h''}{\sqrt{-2}}}\neq \omega_h$, we are done. We shall show that this is the case. 
  
 Let $ U$ be the vertex subalgebra generated by $ \{h, h', h''\}$ in $ V'(h)$ with the conformal vector $ \omega_h$ with has central charge $1-3 a^2$. Then $U$ is a quotient of the affine vertex algebra $ V_{\widehat{\frak{sl}}_2}(1,0)$ and there is a surjective map $ U\rightarrow L_{\widehat{\frak{sl}}_2}(1,0)$ such that the images of $ \omega_{h}$, $ \omega_{h'}$, and $\omega_{h''}$ in $ L_{\widehat{\frak{sl}_2}}(1,0)$ are 
 $$ \tilde{\omega}_{h}=\frac{1}{2}h(-1)h(-1)\mbf{1}-\frac{a}{2}h(-2)\mbf{1},  \tilde{\omega}_{h'}=\frac{1}{2}h'(-1)h'(-1)\mbf{1}-\frac{b}{2}h'(-2)\mbf{1}$$and $$ \tilde{\omega}_{h''}=\frac{1}{2}h''(-1)h''(-1)\mbf{1}-\frac{c}{2}h''(-2)\mbf{1},$$
  respectively. Here we are using the notation $ h(-1)=h\otimes t^{-1}$ in the affine Lie algebra. Note that, under the basis $ \{e, f, k\}$ in $ \frak{sl}_2$, we have 
  \[ h=\frac{\sqrt{2}i}{2}(e-if), \quad  h'=\frac{\sqrt{2}i}{2}(f-ie), \quad \text{ and } \quad h''=-ik.\]
 
  Expressing three vectors $\omega_h,\omega_h',\omega_h''$  in terms of $ \{e, f, k\}$ in $L_{\widehat{\frak{sl}}_2}(1,0)=V_{\widehat{\frak{sl}}_2}(1,0)/\<e(-1)^2\mbf{1}\> $,  we have 
  \begin{eqnarray*}
\tilde{\omega}_{h}&=& \frac{1}{4}(f(-1)f(-1)+if(-1)e(-1)+ie(-1)f(-1))\mbf{1}-\frac{\sqrt{2}ai}{4}e(-2)\mbf{1}+\frac{\sqrt{2}a}{4}f(-2)\mbf{1},\\
\tilde{\omega}_{h'}&=& \frac{1}{4}(-f(-1)f(-1)+if(-1)e(-1)+ie(-1)f(-1))\mbf{1}-\frac{\sqrt{2}bi}{4}f(-2)\mbf{1}+\frac{\sqrt{2}b}{4}e(-2)\mbf{1},\\
 \tilde{\omega}_{\frac{h''}{\sqrt{-2}}}&=&\frac{1}{4}k(-1)k(-1)\mbf{1}+\frac{ci}{2}k(-2)\mbf{1}.
  \end{eqnarray*}
 Hence $ \tilde{\omega}_{h}-\tilde{\omega}_{h'}=\frac{1}{2}f(-1)^2\mbf{1}+\frac{\sqrt{2}}{4}(a-bi)f(-2)\mbf{1}+\frac{\sqrt{2}}{4}(b-ai)e(-2)\mbf{1}\neq 0$ by considering a PBW basis in $V_{\widehat{\frak{sl}}_2}(1,0)$ in the order $ \prod_r f(-n_{r})\prod_s k(-m_{s})\prod_t e(-l_{t}) $ with all $ n_r, m_s, l_t>0$ and using the fact that the weight two subspace of $\<e(-1)^2\mbf{1}\>$ is $ \C e(-1)^2\mbf{1}$. This also implies that  $ \tilde{\omega}_h$,  $ \tilde{\omega}_{h'},$ and $ \tilde{\omega}_{\frac{h''}{\sqrt{-2}}}$ are distinct in $ L_{\widehat{\frak{sl}}_2}(1,0)$.   This contradicts the early assumption that   $\omega_{h}=\omega_{h'}={\omega}_{\frac{h''}{\sqrt{-2}}}$.  Thus, we  have completed the proof of the theorem. \qed
  \end{proofof}   

\begin{remark}
%1) The condition $ L(1)V_1=0$ is satisfied for most vertex operator algebras arising from Lie algebras with standard %conformal structures \cite[7.1.1, 7.1.2]{FB}. 

1) In Theorem 1.6,   the condition $ L(1)V_1=0$ isn't satisfied on $(V_{\hat{\h}}(1,0),\omega_{h})$ 
for $ h\neq 0$.  But we know there exists a $a\in \C$ such that $L(1)h=a\mbf{1}$ for any $h\in V_1$.

2)   The tensor decomposition condition \eqref{e5.5}
\begin{equation*}
\forall \omega'\in \on{Sc}(V,\omega), V\simeq C_{V}(C_{V}(\<\omega'\>))\otimes C_{V}(\<\omega'\>)
\end{equation*}   is    critical in conditions to distinguishing Heisenberg vertex algebra from other vertex algebras arising from Lie algebras.  Checking this condition could be a challenge.  

 For a vertex operator algebra $(V, \omega)$ of  nondegenerated simple CFT type, by Lemma 6.1,  the condition \eqref{e5.5} implies that for an $\omega' \in\on{Sc}(V, \omega)$,   the following assertions hold:
\begin{itemize}
\item[(1)]
$C_{V}(C_{V}(\<\omega'\>))_1$ and $C_{V}(\<\omega'\>)_1$ are mutually orthogonal in $V_1$;

\item[(2)] $V_1\simeq C_{V}(C_{V}(\<\omega'\>))_1\oplus C_{V}(\<\omega'\>)_1.$
\end{itemize}

\end{remark}
On the other hand, the following Lemma will provide a way to verify the condition \eqref{e5.5}. It will also be used in the proof of Theorem~\ref{thm1.7}.

\begin{lem} \label{lem5.7} Let $(V,\omega)$ be a nondegenerate simple CFT type vertex operator algebra.
For $\omega'\in \on{Sc}(V,\omega)$, assume  $\on{dim}C_{V}(C_{V}(\<\omega'\>))_1\neq 0$, $\on{dim}C_{V}(\<\omega'\>)_1\neq 0$,  and $\dim V_1=\dim C_{V}(\<\omega'\>)_1 +\dim C_{V}(C_{V}(\<\omega'\>))_1$, then we have
\begin{itemize}
\item[1)] $V=C_{V}(\<\omega'\>)\otimes C_{V}(C_{V}(\<\omega'\>));$
\item[2)] $C_{V}(\<\omega'\>)=\<C_{V}(\<\omega'\>)_1\> \text{ and} \; C_{V}(C_{V}(\<\omega'\>))=\<C_{V}(C_{V}(\<\omega'\>))_1\>.$
\end{itemize}
\end{lem}
%\begin{proof}
%Note that $<C_{V}(<\omega'>)_1>\otimes<C_{V}(C_{V}(<\omega'>))_1>$ is a vertex operator subalgebra of $V$ with the same conformal vector  $\omega$ and weight one subspace being $ C_{V}(<\omega'>)_1\oplus C_{V}(C_{V}(<\omega'>))_1$. The assumption of the dimensions of weight one subspaces implies 
 %$$C_{V}(<\omega'>)_1\oplus C_{V}(C_{V}(<\omega'>))_1=V_1.$$
%Since $ V$ is generated by $ V_1 $ and the centres of $C_{V}(<\omega-\omega'>)$ and $ C_{V}(<\omega'>)$ are all one dimensional subspace $ \C\mbf{1}$,   then the assertions 1) and 2) will follow from the definition and Lemma~\ref{l6.1}. 
%\end{proof}
%\begin{theorem}
%Let $(V,\omega)$ be a nondegenerate simple CFT type vertex operator algebra. 
%If $\on{\dim}V_1=d$ and there exists a chain 
%$0=\omega^0\prec \omega^1\prec\cdots \prec \omega^{d-1}\prec \omega^{d}
%=\omega$ in $\on{Sc}(V,\omega)$ such that 
%$\on{dim}C_{V}(C_{V}(<\omega^{i}-\omega^{i-1}>))_1\neq 0, ~for~i=1,\cdots, d$. 
%Then  $V$ is isomorphic to the Heisenberg vertex operator algebra 
%$V_{\widehat{V_1}}(1,0)$.
%\end{theorem}

\begin{proofof}{\bf Proof of Theorem 1.8.}
For notational convenience, we denote $U(\omega')=C_{V}(C_{V}(\<\omega'\>))$ for each $\omega'\in \on{Sc}(V,\omega)$.  Then $U(\omega-\omega')=C_{V}(\<\omega'\>)$.
 Note that $U(\omega^1)\otimes U(\omega^2-\omega^1)\otimes\cdots \otimes U(\omega^d-\omega^{d-1})$ is a conformal subalgebra  of $V$ with the same conformal vector $\omega$. We know that
$U(\omega^1)_1\oplus U(\omega^2-\omega^1)_1\oplus \cdots \oplus U(\omega^d-\omega^{d-1})_1$ is a subspace of $V_1$. Since $\on{dim}U(\omega^{i}-\omega^{i-1})_1\neq 0$ for $i=1,\cdots,d$, then $\on{dim}U(\omega^{i}-\omega^{i-1})_1\geq 1$. Hence
$$V_1=U(\omega^1)_1\oplus U(\omega^2-\omega^1)_1\oplus \cdots \oplus U(\omega^d-\omega^{d-1})_1.$$  Thus $ \dim U(\omega^i-\omega^{i-1})_1=1$.  By  the given chain condition  and an induction on $d$  using  Lemma~\ref{lem5.7}, we have  $$V=\<V_1\>=\<U(\omega^1)_1\>\otimes \<U(\omega^2-\omega^1)_1\>\otimes\cdots \otimes \<U(\omega^d-\omega^{d-1})_1\>
.$$  Thus $V=U(\omega^1)\otimes U(\omega^2-\omega^1)\otimes\cdots \otimes U(\omega^d-\omega^{d-1})$ and  $U(\omega^{i}-\omega^{i-1})=\<U(\omega^{i}-\omega^{i-1})_1\>$ for $i=1,\cdots,d$. Since  $\on{dim}U(\omega^{i}-\omega^{i-1})_1=1$, we take $h^1\in U(\omega^{i}-\omega^{i-1})_1$ and $\langle h^1,h^1\rangle=1$ such that $L(1)h^1=\lambda_1\mathbf{1}$, Proposition~\ref{p5.4} implies that   $U(\omega^{i}-\omega^{i-1})$ is isomorphic to a Heisenberg vertex operator algebra $V_{\hat{\h}}(1,0)$ with $\dim \h =1$ and conformal vector of the form $ \omega _{h^1}=\frac{1}{2}h^1_{-1}h^1_{-1}\mbf{1}-\frac{\lambda_1}{2}h^1_{-2}\mbf{1}$.  Therefore, $V$ is isomorphic to the Heisenberg vertex operator algebra $V_{\widehat{\h}}(1,0)$ with $ \h=V_1$ and the conformal vector of the form $ \omega_{h^1}$. \qed
\end{proofof}
\begin{remark} Theorem~\ref{thm1.7} can also be proved using the argument used in the proof of  Theorem~\ref{thm1.6}. The advantage of Theorem~\ref{thm1.7} is that one does not have to verify the condition \eqref{e5.5} for all $ \omega' \in \on{Sc}(V,\omega)$.  The condition of Theorem~\ref{thm1.7} is more convenient to verify when one uses a specific construction of $V$ such as most of the known examples.  When the condition $ L(1)V_1=0$ is removed, we prove that $V$ is still isomorphic to $ V_{\hat{\h}}(1, 0)$ with conformal vector $ \omega_{h}$ and $ h$ needs not be zero.
\end{remark}
\begin{remark}
Our aim is to  understand the invariants  of a vertex  algebra $V$ in terms of the moduli space of conformal vectors preserving a fixed gradation of $V$.  In particular, we first understand the invariants of a vertex operator algebra $(V,\omega)$ by virtue of the moduli space of semi-conformal vectors of $(V,\omega)$.
 As the Heisenberg case has indicated here,  for an affine vertex operator algebra $V$ associated to a finite dimensional simple Lie algebra $ \frak{g}$,  the variety $\on{Sc}(V, \omega)$ is closely related  to the  Lie algebra $\frak{g}$ with the fixed level $k$(See \cite{CL2} for the case of Lie algebra $\sl_2(\C)$).  In the case that $(V, \omega)$ is a lattice vertex operator algebra associated to a lattice $ (L, \langle\cdot, \cdot \rangle )$, the variety $\on{Sc}(V, \omega)$ is determined completely by the lattice structure.   We shall  describe invariants of affine  vertex operator algebras   in terms of  their moduli spaces of conformal vectors preserving a fixed gradation and semi-conformal vectors .
\end{remark}

\end{document}